%% file: final_main.tex
\documentclass[a4paper,11pt]{article}

\usepackage{adjustbox}

\newcommand{\be }{\begin{equation}}
\newcommand{\ee }{\end{equation}}

\newcommand{\dd}{\mathrm{d}}

\input{final_preamble/packages}
\input{final_preamble/config}

\input{final_preamble/theorems}
\input{final_preamble/commands}

\input{final_preamble/notations}
\input{final_preamble/specialized}

\title{Stokes phenomenon and quantum supergroup $U_q(\mathfrak{gl}(m|n))$}
\author{Qiao Li, Zikang Wang and Xiaomeng Xu}
\date{}

\newcommand{\Addresses}{{
  \bigskip
  \footnotesize
\noindent \textsc{
School of Mathematical Sciences \& Beijing International Center
for Mathematical Research, Peking University, Beijing 100871, China}\par\nopagebreak
  \textit{E-mail address}: \texttt{xxu@bicmr.pku.edu.cn}
}}

\newcommand{\AddressesLi}{{
  \bigskip
  \footnotesize
\noindent \textsc{
School of Mathematical Sciences, Peking University, Beijing 100871, China}\par\nopagebreak
  \textit{E-mail address}: \texttt{lqlgw\_pkusms@stu.pku.edu.cn}
}}

\newcommand{\AddressesWang}{{
  \bigskip
\footnotesize
\noindent \textsc{School of Mathematical Sciences, Peking University, Beijing 100871, China}\par\nopagebreak
  \textit{E-mail address}: \texttt{wzikang@stu.pku.edu.cn}
}}

\begin{document}
\maketitle
\begin{abstract}
In this paper we study the Stokes phenomenon of the quantum confluent hypergeometric supersystem, certain meromorphic linear system of ordinary differential equation with a second order pole, associated to the Lie superalgebra $\mathfrak{gl}_{m|n}$. We prove that its Stokes supermatrices satisfy the Yang-Baxter equation, and thus give rise to the quantum supergroup $U_q(\mathfrak{gl}(m|n))$.
\end{abstract}

\input{final_sections/1Introduction}

\input{final_sections/2StokesMatrices}

\input{final_sections/3RLLrelation}

\input{final_sections/4Representation}

\end{document}

%% file: final_preamble/packages.tex
\usepackage[utf8]{inputenc}
\usepackage[T1]{fontenc}
\usepackage{lmodern}
\usepackage{times}
\usepackage{microtype}
\usepackage{amsmath}
\usepackage{amssymb}
\usepackage{amsfonts}
\usepackage{amsthm}
\usepackage{amscd}
\usepackage{mathtools}
\usepackage{mathdots}
\usepackage{bm}
\usepackage{latexsym}

\usepackage{stmaryrd}
\SetSymbolFont{stmry}{bold}{U}{stmry}{m}{n}
\usepackage{mathrsfs}
\usepackage{bbm}
\usepackage{xfrac}
\usepackage{nicefrac}

\usepackage{graphicx}
\usepackage{float}
\usepackage{subfigure}
\usepackage{tikz}
\usetikzlibrary{matrix,arrows,decorations.pathmorphing}
\usepackage{tikz-cd}
\usepackage[all]{xy}

\usepackage{enumitem}

\usepackage{geometry}
\usepackage{hyperref}
\usepackage[toc,page]{appendix}

\usepackage[usestackEOL]{stackengine}
\usepackage{color}
\usepackage{lipsum}
\usepackage{amsxtra}

\usepackage{comment}

%% file: final_preamble/config.tex
\usepackage{bm}
\SetSymbolFont{stmry}{bold}{U}{stmry}{m}{n}

\def\qed{\hfill ~\vrule height6pt width6pt depth0pt}

%% file: final_preamble/theorems.tex
\usepackage{theoremref}
\newtheorem{thm}{Theorem}[section]
\newtheorem{lem}[thm]{Lemma}

\newtheorem{pro}[thm]{Proposition}

\theoremstyle{definition}

\newtheorem{rmk}[thm]{Remark}
\newtheorem{defi}[thm]{Definition}

%% file: final_preamble/commands.tex
\newcommand{\IC}{\mathbb{C}}

\newcommand{\h}{\mathfrak h}

%% file: final_preamble/notations.tex
\newcommand{\End}{\mathrm{End}}

\newcommand{\ad}{\operatorname{ad}}

\newcommand{\Sect}{{\rm Sect}}

\newcommand{\adjoint}{\mathrm{ad}}

\newcommand{\mathi}{\mathrm{i}}

\newcommand{\rednote}[1]{\textcolor{red}{#1}}

\newcommand{\br}[1]{   [ \cdot,    \cdot  ]   }

%% file: final_preamble/specialized.tex
\newcommand {\hreg}{\mathfrak h_{\rm reg}}

\newcommand{\bi}{\bar{i}}
\newcommand{\bj}{\bar{j}}
\newcommand{\bk}{\bar{k}}
\newcommand{\bl}{\bar{l}}

%% file: final_sections/1Introduction.tex
\section{Introduction}\label{sec:1}
The Stokes phenomenon of meromorphic linear systems of ordinary differential equations with a second order pole has been intensively studied over the past decades, both for its own analytic interest and for its deep connections with other subjects, like the theory of isomonodromy deformation \cite{JMU}, the Gromov-Witten theory \cite{Dubrovin}, the stability condition \cite{Bridgeland} and so on.

More closely related to our present paper is the connections to Poisson Lie groups and quantum groups, briefly recalled as follows. 
%On the analytic side, the theory of classical ODEs and the multisummability theory of Malgrange and Ramis~\cite{MR}, Wasow \cite{Wasow}, Balser~\cite{Balser} and Loday-Richaud~\cite{LR} provide a complete description of the canonical sectorial solutions and their Stokes matrices. Moreover, the fundamental work of Jimbo, Miwa and Ueno~\cite{JMU} endows the Stokes matrices the role of monodromy data of irregular systems that can be studied under the framework of isomodromic deformations.
In \cite{Boalch1} Boalch discovered that the space of Stokes matrices is identified with the dual Poisson Lie groups, and proved that the Riemann-Hilbert-Birkhoff map is a Poisson map. Later on, in \cite{Boalch2,Boalch4} he developed the Poisson geometry on the moduli spaces of meromorphic connections with arbitrary order poles (not just second order).
In \cite{TL2016} Toledano Laredo  studied the Stokes phenomenon of the Knizhnik-Zamolodchikov (KZ) equation with second order poles. He used the connection matrix
of the KZ equation to construct a twist killing the KZ associator, and thus obtain a quasi-triangular Hopf algebra that is equivalent to the Drinfeld-Jimbo quantum groups. In \cite{TLX2023} Toledano Laredo and the third author introduced the quantum Stokes matrices of the KZ equation with a second order pole, and proved that the semiclassical limit of the quantum Stokes matrices coincides with the Stokes matrices of the (classical) meromorphic linear systems, thereby providing a quantization of the Riemann-Hilbert-Birkhoff map for the case of second order pole. 

Motivated by the above works, in \cite{Xu, Xu2019} the third author studied the Stokes phenomenon and WKB approximation of the quantum confluent hypergeometric system associated to $\mathfrak{gl}_n$. He proved that the Stokes matrices give rise to the Faddeev--Reshetikhin--Takhtajan (FRT) realization of the quantum group $U_q(\mathfrak{gl}_n)$, and the WKB approximation of the Stokes matrices give rise to the $\mathfrak{gl}_n$-crystals. Then in \cite{TX2025}, using the theory of difference equations, Tang and the third author realized the Stokes matrices as infinite products of the representations of Yangian, and thus gave a new interpretation why quantum groups arise from the Stokes phenomenon. Recently, the third author \cite{Xu2026} studies the Stokes phenomenon of the quantum meromorphic linear system, certain KZ type equations, with arbitrary order poles, generalizing the results in \cite{Xu} from the case of second order poles to arbitrary orders.

However, as far as we know, the Stokes phenomenon for super (i.e.,~$\mathbb{Z}_2$-graded) meromorphic linear systems has not been investigated, let alone its connection with quantum supergroups, despite the fact that superalgebras play an important role in mathematical physics. 

In this paper, we introduce the Stokes (super)matrices of the quantum confluent hypergeometric supersystem, and prove that they give rise to the quantum supergroup $U_q(\mathfrak{gl}(m|n))$.

\medskip
\noindent\textbf{The quantum confluent hypergeometric supersystem.}
Let $\mathfrak{gl}_{m|n}$ be the general linear Lie superalgebra over the complex numbers $\mathbb{C}$. Let $U(\mathfrak{gl}_{m|n})$ be its universal enveloping algebra, with generators $\{e_{kl}\}_{k,l=1,...,m+n}$ and relations
\begin{align}
    e_{ij}\cdot e_{kl}-(-1)^{(\bar{i}+\bar{j})(\bar{k}+\bar{l})}e_{kl}\cdot e_{ij}=\delta_{jk}e_{il}-(-1)^{(\bar{i}+\bar{j})(\bar{k}+\bar{l})}\delta_{li}e_{kj}.
\end{align} 
Take any finite dimensional indecomposable representation $V$ of $U(\mathfrak{gl}_{m|n})$, consider the following meromorphic linear differential system for a function $F(z)$ valued in the 
tensor product $\End(V)\otimes\End(\mathbb{C}^{m|n})$ of superalgebras:

\begin{align}\label{eq:MonoSingularSys}
\frac{\mathrm{d} F}{\mathrm{d} z}= h\left(u+\frac{T}{z}\right) F(z),
\end{align}
where $h$ is a complex parameter, 
\begin{align*}
u=\sum_{i=1}^{m+n}1\otimes u_iE_{ii}\in \End(V) \otimes \End(\IC^{m|n}),
\quad \text{with } u_1,...,u_{m+n} \text{ distinct complex numbers},
\end{align*}
and
\begin{align*}
T=\sum_{k,l=1}^{m+n}(-1)^{\bar k\bar l}e_{kl}\otimes E_{kl} \in \End(V) \otimes \End(\IC^{m|n}).
\end{align*}
Here $E_{ij}\in \End(\mathbb C^{m|n})$ are the matrix units, $e_{kl}\in \End(V)$ is the image of $e_{kl}$ under the given representation.

System~\eqref{eq:MonoSingularSys} has an irregular singularity of Poincar\'e rank $1$ at $z=\infty$ and a regular singularity at $z=0$. For $h\notin\mathbb{Q}$, it has a unique formal fundamental solution $\widehat{F}(z)$ which is in general divergent, see Proposition \ref{prop:formal-solution-sec2}. The standard Borel--Laplace summation theory~\cite{Balser,LR,Wasow} carries to the super setting, since the (super)algebra $\mathcal{A}_V:=\mathrm{End}(V)\otimes\mathrm{End}(\mathbb{C}^{m|n})$ is a Banach algebra and the Koszul signs do not affect analytic estimates. Consequently, for any admissible direction $d$, there exists an unique (therefore canonical) holomorphic fundamental solution $F_d(z)$ with prescribed asymptotics $\widehat{F}(z)$, as $z\rightarrow\infty$ within a sector of opening $\pi$ centered at the direction $d$, see Proposition \ref{prop:canonical-solution-sec2}. One defines the Stokes supermatrices $S_d^{\pm}(u)\in\mathcal{A}_V$ as the transition matrices between the canonical solutions associated to admissible directions $d-\pi,d$ and $d+\pi$:
\begin{equation}
    F_d(z)=F_{d\pm\pi}(z)\cdot S_d^{\pm}(u).
\end{equation}

\medskip
\noindent\textbf{Main results.}
Our main result states that the Stokes (super)matrices of \eqref{eq:MonoSingularSys} gives rise the FRT realization of $U_q(\mathfrak{gl}(m|n))$, see Theorem \ref{thm:RLL-relation}. An
equivalent statement is as follows. 
%let $R\in\mathrm{End}(\mathbb{C}^{m|n})\otimes\mathrm{End}(\mathbb{C}^{m|n})$ be the standard super $R$-matrix of $U_q(\mathfrak{gl}(m|n))$ with $q=e^{\pi\mathrm{i}h}$ (cf. \cite{BrackenGouldZhang1990, Khoroshkin1991}, see \eqref{eq:standard-r-matrix} for its expression), and 
Let
\begin{equation}\label{eq:intro-renorm-stokes}
S_{h+}(u):=S_d^+(u)\,e^{\pi\mathrm{i}h\delta T},\qquad
S_{h-}(u):=e^{-\pi\mathrm{i}h\delta T}\,S_d^-(u)
\end{equation}
be the renormalized Stokes supermatrices by adding the formal monodromy part $e^{\pi\mathrm{i}h\delta T}$ (see Definition~\ref{def:renorm-Stokes}). Let us write
\begin{equation*}
    S_{h\pm}(u)=\sum_{i,j=1}^{m+n}S_{h\pm}(u)_{i,j}\otimes E_{ij}\in \End(V)\otimes \End(\mathbb{C}^{m|n}).
\end{equation*}

\begin{thm}\label{thm:intro-Rep}
For any $h\notin\mathbb{Q},\,u\in\h_{\rm reg}$, let $q=e^{\pi\mathi h}$, set $d_i=(-1)^{\bar{i}}$ and $q_i=q^{d_i}$ for $i\!=\!1,\cdots m\!+\!n$, the map
\begin{align*}
\mathcal{S}_q(u):U_q(\mathfrak{gl}(m|n))\longrightarrow \mathrm{End}(V),
\end{align*}
where
\begin{align*}
e_i\longmapsto \frac{S_{h-}(u)_{i,i}\cdot S_{h+}(u)_{i,i+1}}{q_i^{-1}-q_i},\qquad
f_i\longmapsto \frac{S_{h-}(u)_{i+1,i}\cdot S_{h+}(u)_{i,i}}{q_i-q_i^{-1}},\qquad
K_j\longmapsto \big(S_{h+}(u)_{j,j}\big)^{d_j},
\end{align*}
is a representation of the Drinfeld--Jimbo quantum supergroup $U_q(\mathfrak{gl}(m|n))$ on the vector space $V$ (see Definition \ref{def:DJ-quantumSupergroup} for the definition and generators $(e_i,f_i,K_i)$ of $U_q(\mathfrak{gl}(m|n))$).
\end{thm}

\begin{comment}
\begin{thm}\label{thm:intro-RLL}
For any $h\notin\mathbb{Q}$ and $u\in\mathfrak{h}_{\mathrm{reg}}$, the
Stokes supermatrices satisfy the following identities in $\mathrm{End}(V)\otimes\mathrm{End}(\mathbb{C}^{m|n})\otimes\mathrm{End}(\mathbb{C}^{m|n})$:
\begin{equation}\label{eq:intro-RLL-1}
R^{12} S_{h\pm}^{(1)}(u) S_{h\pm}^{(2)}(u)
= S_{h\pm}^{(2)}(u) S_{h\pm}^{(1)}(u) R^{12},
\end{equation}
\begin{equation}\label{eq:intro-RLL-2}
R^{12} S_{h+}^{(1)}(u) S_{h-}^{(2)}(u)
= S_{h-}^{(2)}(u) S_{h+}^{(1)}(u) R^{12},
\end{equation}
where the superscripts $(1),(2)$ {\color{red}denote the tensor factors on which the operators act in the super sense}.
\end{thm}
Since the construction is universal, the theorem shows that the Stokes supermatrices give rise to the Faddeev--Reshetikhin--Takhtajan (FRT) formalism of the quantum supergroup $U_q(\mathfrak{gl}(m|n))$. It generalizes the main theorem of~\cite{Xu} from the Lie algebra $\mathfrak{gl}_n$ to the Lie superalgebra $\mathfrak{gl}(m|n)$.
\end{comment}
\begin{rmk}
If $h=h_0\in\mathbb{Q}$, then the corresponding differential equation may become resonant. In this case, the uniqueness of the formal fundamental solution (see Proposition \ref{prop:formal-solution-sec2}) is not valid. Instead, the equation \eqref{eq:MonoSingularSys} has a family of formal solution depending on a finite set $c$ of complex parameters. Accordingly, there are a family of Stokes matrices $S_{h_0\pm}(u;c)$ depending on the same set of parameters. 
However, there is a pair of distinguished Stokes matrices, denoted by $S_{h_0\pm}^o(u)$, among the family $S_{h_0\pm}(u;c)$. They are actually the continuous extension of $S_{h\pm}(u)$ from $h\in\mathbb{C}\setminus \mathbb{Q}$ to $h_0\in\mathbb{Q}$.Therefore, for $h=h_0\in\mathbb{Q}$, Theorem \ref{thm:intro-Rep} holds true provided taking the Stokes matrices $S_{h_0\pm}^o(u)$. See \cite{Xu} for a detailed discussion for the case of Lie algebra $\mathfrak{gl}_n$.
\end{rmk}

Theorem \ref{thm:intro-Rep} generalizes the main result of~\cite{Xu} from the Lie algebra $\mathfrak{gl}_n$ to the Lie superalgebra $\mathfrak{gl}_{m|n}$. It states that, starting from a representation of the Lie superalgebra, one obtains a representation of the corresponding quantum supergroup by taking the Stokes supermatrices. The construction is universal, thus providing a realization of $U_q(\mathfrak{gl}(m|n))$ via the Stokes phenomenon. We would like to highlight that, when $V$ is taken to be the natural representation of the Lie superalgebra, the Stokes supermatrices of the equation \eqref{eq:MonoSingularSys} can be computed explicitly and coincide precisely with the standard $R$-matrix of the quantum supergroup, see Theorem \ref{thm:super-r-matrix}. This computation clearly reveals the influence of odd variables on the process of solving the differential equation. 

In future work, we plan to deepen the relation between the Stokes phenomenon in the presence of odd variables and the representation theory of Lie superalgebras. In particular, we plan to study correspondences between the long time asymptotics of isomonodromy deformation equations and the Gelfand–Tsetlin basis, Stokes supermatrices and super Yangians, and the WKB approximation and crystals for quantum supergroups, thereby generalizing the results of \cite{TX2025, Xu2019} to the super setting. These generalizations will be more than just parallel extensions. Rather, we expect that, on the one hand, they will make more essential use of the representation theory of Lie superalgebras, and on the other hand, new structures will emerge in the Stokes phenomenon when odd variables come into play. For example, we can imagine that when odd variables are present, the formal divergent power series solutions of differential equations undergo truncation in the odd variables, potentially altering the structures in the resummation theory.

\medskip
\noindent\textbf{Organization of the paper.}
Section~\ref{sec:2} introduces the canonical solutions and the Stokes (super)matrices of the equation \eqref{eq:MonoSingularSys}. It then computes explicitly the Stokes matrices for the $V=\mathbb{C}^{m|n}$ case, and identifies them with the
standard super $R$-matrix for quantum supergroup. Section~\ref{sec:3} proves Theorem~\ref{thm:RLL-relation} (the FRT realization of $U_q(\mathfrak{gl}(m|n))$ via Stokes supermatrices), by studying the Stokes phenomenon and connection formula for the super irregular KZ equations. In Section \ref{sec:4}, the $RLL=LLR$ relation is used to prove Theorem \ref{thm:intro-Rep}.

%% file: final_sections/2StokesMatrices.tex
\section[Stokes matrices of the equation in the nonresonant case]%
        {Stokes matrices of the equation \eqref{eq:MonoSingularSys} in the nonresonant case}%
\label{sec:2}

In this section we study the Stokes phenomenon of the superalgebra-valued equations \eqref{eq:MonoSingularSys}.
\begin{comment}
\begin{equation}\label{eq:main-system-sec2}
\frac{dF}{dz}
=
h\Bigl(u+\frac{T}{z}\Bigr) F,
\end{equation}
where $V$ is a fixed finite-dimensional $U(\mathfrak{gl}(m|n))$-module, as defined in the introduction. 
\end{comment}
We first introduce the concept of (Lie) superalgebras and describe the Banach algebra where our equations take values. Then, using classical analytical techniques, we define the canonical solutions of the equations and analyze their Stokes phenomenon in this Banach space.

\subsection{The normed super algebra \texorpdfstring{$\mathcal A_V$}{A\_V}}\label{sec:2.1}

We begin by briefly recalling the superalgebra conventions used throughout this paper. For more details, see e.g. \cite{Kac1977, Manin1997,Carmeli2011}. A \emph{super vector space}
is a $\mathbb Z_2$-graded vector space $W=W_{\bar 0}\oplus W_{\bar 1}$. An element $w\in W$ is called
\emph{homogeneous} if $w\in W_{\bar 0}$ or $w\in W_{\bar 1}$; in that case we write $|w|\in \mathbb Z_2$ for its
parity. A \emph{superalgebra} is an associative algebra $A=A_{\bar 0}\oplus A_{\bar 1}$ whose multiplication
respects the grading, namely $A_{\bar a}A_{\bar b}\subset A_{\bar a+\bar b}$ for $\bar a,\bar b\in \mathbb Z_2$.
For homogeneous elements $x,y\in A$, their super commutator is defined by
\begin{equation}\label{eq:super-commutator-sec2}
[x,y]
:=
xy-(-1)^{|x||y|}yx.
\end{equation}

The natural superspace is $\mathbb{C}^{m|n} = \mathbb{C}^{m|n}_{\bar{0}} \oplus \mathbb{C}^{m|n}_{\bar{1}}$, where $\dim \mathbb{C}^{m|n}_{\bar{0}} = m$ and $\dim \mathbb{C}^{m|n}_{\bar{1}} = n$. Any super vector space $V = V_{\bar{0}} \oplus V_{\bar{1}}$ with $\dim V_{\bar{0}} = m$ and $\dim V_{\bar{1}} = n$ is isomorphic to $\mathbb{C}^{m|n}$ upon choosing a homogeneous basis. 
The space of linear endomorphisms $\End(V)$ naturally inherits a $\mathbb{Z}_2$-grading, $\End(V) = \End(V)_{\bar{0}} \oplus \End(V)_{\bar{1}}$, defined by
\[
\End(V)_{\bar{a}} := \{ f \in \End(V) \mid f(V_{\bar{b}}) \subset V_{\bar{a}+\bar{b}} \text{ for all } \bar{b} \in \mathbb{Z}_2 \}.
\]
Equipped with the composition of mappings as multiplication, $\End(V)$ becomes a superalgebra. 
\begin{comment}
Moreover, via the chosen homogeneous basis, the endomorphism algebra $\End(V)$ is identified with the standard matrix superalgebra $\text{M}(m|n, \mathbb{C})$. Under this identification, an element $X \in \End(\mathbb{C}^{m|n})$ can be represented as an $(m+n) \times (m+n)$ block matrix
\begin{equation}\label{eq:block_matrix}
X = \begin{pmatrix} A & B \\ C & D \end{pmatrix},
\end{equation}
where $A, B, C, D$ are matrices of sizes $m \times m$, $m \times n$, $n \times m$, and $n \times n$, respectively. $X \in \End(\mathbb{C}^{m|n})_{\bar{0}}$ if and only if $B=0$ and $C=0$, while $X \in \End(\mathbb{C}^{m|n})_{\bar{1}}$ if and only if $A=0$ and $D=0$.
\end{comment}
More precisely, let us fix a homogeneous basis $\{e_1, \dots, e_{m+n}\}$ of $\mathbb{C}^{m|n}$ such that $e_1, \dots, e_m$ are even and $e_{m+1}, \dots, e_{m+n}$ are odd. We write $\bar{i} := |e_i| \in \mathbb{Z}_2$ for the parity of the basis vector $e_i$. Under this basis, the endomorphism algebra $\End(\mathbb{C}^{m|n})$ is spanned by the matrix units $E_{ij} \in \End(\mathbb{C}^{m|n})$, whose action on the basis vectors is given by $E_{ij}e_k = \delta_{jk}e_i$. The $\mathbb{Z}_2$-grading of the matrix units is then explicitly given by their parities $|E_{ij}| = \bar{i} + \bar{j}$.

The Lie superalgebra $\mathfrak{gl}_{m|n}$ is generated by the standard elements $\{e_{ij}\}_{1\le i,j\le m+n}$. The $\mathbb{Z}_2$- grading of $e_{ij}$ is $\bar i+\bar j$, where $\bar{i} $ equals to $0$ if $i\leq m$ and $1$ if $i>m$. The commutation relations in this basis are given by
\begin{equation}\label{eq:glmn-super-rel-sec2}
[e_{ij},e_{kl}]
=
\delta_{jk}e_{il}
-
(-1)^{(\bar i+\bar j)(\bar k+\bar l)}\delta_{li}e_{kj}.
\end{equation}
Let $V$ be a fixed finite-dimensional indecomposable $U(\mathfrak{gl}_{m|n})$-module. As in the rest of the paper, we
suppress the representation map $\rho:U(\mathfrak{gl}_{m|n})\to \End(V)$ from the notation, and regard the
generators $e_{ij}$ as endomorphisms of $V$.

Once the representation $V$ is fixed, all analytic constructions will take place in the finite-dimensional algebra
\begin{equation}
\mathcal{A}_V := \End(V) \otimes \End(\mathbb{C}^{m|n}).
\end{equation}
Every element of $\mathcal{A}_V$ can be written uniquely as a linear combination $A = \sum_{i,j} A_{ij} \otimes E_{ij}$ with $A_{ij} \in \End(V)$. The tensor product here (and throughout this paper) is also a superalgebra: for homogeneous elements $a, a' \in \End(V)$ and $b, b' \in \End(\mathbb{C}^{m|n})$, one sets
\begin{equation}\label{eq:super_tensor}
(a \otimes b)(a' \otimes b') = (-1)^{|b||a'|} aa' \otimes bb'.
\end{equation}

\begin{comment}
Accordingly, every element of $\mathcal{A}_V$ can be written uniquely as a linear combination $A = \sum_{i,j} A_{ij} \otimes E_{ij}$ with $A_{ij} \in \End(V)$. In practical computations, it is convenient to identify the operator $A = \sum_{i,j} A_{ij} \otimes E_{ij}$ with the matrix whose $(i,j)$-entry is $\tilde{A}_{ij} = A_{ij} (-1)^{\bar{j}(\bar{i}+\bar{1})}$. This sign modification directly allows us to compute the super-tensor product using ordinary matrix multiplication for $(\tilde{A}_{ij})$.
\end{comment}

\begin{comment}
and may therefore be viewed as an
$(m+n)\times(m+n)$ matrix with entries in $\End(V)$. If
$A=\sum_{i,j}A_{ij}\otimes E_{ij}$ and $B=\sum_{i,j}B_{ij}\otimes E_{ij}$, then their product in $\mathcal A_V$
takes the form
\begin{equation}\label{eq:supermatrix-product-sec2}
(A\cdot_s B)_{ij}
=
\sum_{k=1}^{m+n}
(-1)^{|E_{ik}||B_{kj}|}\,A_{ik}B_{kj}.
\end{equation}
We shall use the notation $\cdot_s$ for the product \eqref{eq:supermatrix-product-sec2} and refer to it as the \rednote{supermatrix} product, and simply $\cdot$ for the ordinary product when no ambiguity can arise. The two products differ only by the Koszul signs coming from the super tensor product structure. In particular, if one factor consists entirely of even elements (such as a diagonal matrix with even entries), then its product with any $\mathcal A_V$-valued function coincides with the ordinary matrix product. This convention will be used repeatedly below.
\end{comment}

In the paper we use $1$ to denote the unit element (identity endomorphism) in $\End(V)$ and $I$ to denote the unit element $\sum_i 1\otimes E_{ii}$ in $\mathcal{A}_V=\End(V)\otimes \End(\mathbb{C}^{m|n})$.

We now choose a norm on $\mathcal A_V$ adapted to the supermatrix realization above. Fix any norm
$|\cdot|_V$ on the finite-dimensional vector space $V$, and let $\|\cdot\|_{\End(V)}$ be the induced operator
norm on $\End(V)$. For an element $x=\sum_{j=1}^{m+n}x_j\otimes e_j\in V\otimes \mathbb C^{m|n}$, define
\begin{equation}\label{eq:vector-norm-sec2}
|x|_V:=\max_{1\le j\le m+n}|x_j|_V.
\end{equation}
For an element $A=\sum_{i,j=1}^{m+n}A_{ij}\otimes E_{ij}\in \mathcal A_V$, define
\begin{equation}\label{eq:algebra-norm-sec2}
\|A\|
:=
\max_{1\le i\le m+n}\sum_{j=1}^{m+n}\|A_{ij}\|_{\End(V)}.
\end{equation}

The function $\|\cdot\|$ is a norm on $\mathcal A_V$, and for any $A,B\in \mathcal A_V$ and
$x\in V\otimes \mathbb C^{m|n}$ one has
\begin{equation}\label{eq:submultiplicative-sec2}
\|A\cdot B\|\le \|A\|\,\|B\|,
\end{equation}
and
\begin{equation}\label{eq:action-estimate-sec2}
|A x|_V\le \|A\|\,|x|_V.
\end{equation}
Thus $\mathcal A_V$ is a finite-dimensional normed algebra, and all analytic estimates in the sequel can be
carried out  as in the classical matrix-valued case. Moreover, as all norms on finite-dimensional spaces are
equivalent, the constructions below do not depend on the particular choice of $|\cdot|_V$.

\subsection{The unique formal fundamental solution in the nonresonant case}\label{sec:2.2}

By analogy with the classical meromorphic linear systems with an irregular singularity at infinity, we look for a
formal fundamental solution of \eqref{eq:MonoSingularSys} in the form
\begin{equation}\label{eq:formal-solution-sec2}
\widehat F(z)=\widehat H(z)\cdot e^{huz}z^{h\delta T},
\qquad
\widehat H(z)=I+\sum_{r\ge1}H_r z^{-r},\quad H_r\in \mathcal A_V.
\end{equation}
Here 
\begin{equation}\label{eq:delta-T-sec2}
\delta T:=\sum_{i=1}^{m+n}(-1)^{\bar i}e_{ii}\otimes E_{ii}\in \mathcal A_V, 
\end{equation}
is the diagonal part of $T$. Moreover, $u=\sum_i1\otimes u_iE_{ii}$ commutes with $\delta T$. For a diagonal element $D=\sum_i D_i\otimes E_{ii}\in \mathcal A_V$, we write
\[
\exp(D):=\sum_i e^{D_i}\otimes E_{ii}.
\]
In particular, on the universal cover $\widetilde{\mathbb C^\times}$ we set
\[
z^{h\delta T}:=\exp\bigl(h(\log z)\delta T\bigr).
\]

Thus the problem is reduced to determining the formal series $\widehat H(z)$. The following proposition shows
that in the nonresonant case this formal solution exists and is uniquely defined.

\begin{pro}\label{prop:formal-solution-sec2}
Assume that $h\notin\mathbb Q$ and $u_i\neq u_j$ for $i\neq j$. Then \eqref{eq:MonoSingularSys} admits a unique
formal fundamental solution of the form \eqref{eq:formal-solution-sec2}. Moreover, let
\[
H_r=\sum_{i,j=1}^{m+n}H_{r,ij}\otimes E_{ij},
\qquad H_{r,ij}\in \End(V),
\]
then each $H_{r,ij}$ is homogeneous of parity $\bar i+\bar j$.
\end{pro}

\begin{proof}
Substituting \eqref{eq:formal-solution-sec2} into \eqref{eq:MonoSingularSys} and comparing the coefficients of
$z^{-r-1}$, we obtain
\begin{equation}\label{eq:recursion-sec2}
[H_{r+1},u]
=
\frac{r}{h}H_r+T\cdot H_r-H_r\cdot\delta T,
\qquad r\ge0.
\end{equation}

We prove simultaneously, by induction on $r$. First It is clear that when $r=0$ $H_0=I$.
Assume that $H_r=\sum_{i,j}H_{r,ij}\otimes E_{ij}$ is already known and that each $H_{r,ij}$ has parity
$\bar i+\bar j$. Then the $(i,j)$-entry of the right-hand side of \eqref{eq:recursion-sec2} is homogeneous of
parity $\bar i+\bar j$. Indeed, the term $\frac{r}{h}H_{r,ij}$ has this parity, each summand of
$(T\cdot H_r)_{ij}$ has parity
$
(\bar i+\bar k)+(\bar k+\bar j)=\bar i+\bar j,
$
and $(H_r\delta T)_{ij}$ has the same parity because $\delta T$ is even diagonal.

Now consider the $(i,j)$-entry of \eqref{eq:recursion-sec2}. Since $u$ is diagonal, it is equal to
\begin{align}
\label{eq:off-diagonal-sec2}
(u_j-u_i)H_{r+1,ij}
=\frac{r}{h}H_{r,ij}
+
\sum_{l\neq j}
(-1)^{\bar i\bar j+\bar l+\bar l\bar j}\,
e_{il}H_{r,lj}
+(-1)^{\bar{i}\bar{j}}e_{ij}
H_{r,jj}
-
(-1)^{\bar j}H_{r,ij}e_{jj}
\end{align}
For $i\neq j$, the scalar $u_j-u_i$ is nonzero, so this determines $H_{r+1,ij}$ uniquely. Moreover,
$H_{r+1,ij}$ is homogeneous of parity $\bar i+\bar j$. For $i=j$, the left-hand side vanishes and one obtains an equation of the form
\begin{equation}\label{eq:diagonal-sec2}
\Bigl(\frac{r}{h} \mathrm{Id}+(-1)^{\bar i}\mathrm{ad}_{e_{ii}}\Bigr)H_{r,ii}=\sum_{l\ne i}
(-1)^{(\bar i+1)(\bar l+1)}\,
e_{il}H_{r,li},
\end{equation}
where $\mathrm{ad}_xy=[x,y]$. Since $V$
is finite-dimensional indecomposable, the eigenvalues of $\mathrm{ad}_{e_{ii}}$ are {integers}. Because $h\notin\mathbb Q$, the operator $\frac{r}{h}\mathrm{Id}+(-1)^{\bar i}\mathrm{ad}_{e_{ii}}$ is invertible for every $r\ge0$. Hence \eqref{eq:diagonal-sec2} determines $H_{r+1,ii}$ uniquely. As
the right side of \eqref{eq:diagonal-sec2} is even, so is $H_{r+1,ii}$, namely it has parity $\bar i+\bar i=0$.

This completes the induction and proves the proposition.
\end{proof}

From  the recurrence relation \eqref{eq:recursion-sec2} we see that the coefficients $H_r$ grow at a factorial rate of $r!$. Consequently, the formal series $\widehat{H}$ is of Gevrey order $1$, which allows one to apply the Borel-Laplace  resummation technique to obtain its sectorial analytic solutions.

\begin{comment}
\begin{lem}\label{lem:gevrey-sec2}
There exists a constant $C>0$ such that
\[
\|H_r\|\le C^{\,r+1}r!,
\qquad r\ge0.
\]
In particular, $\widehat H(z)$ is a Gevrey-$1$ formal series.
\end{lem}

\begin{proof}
Let $\Delta:=\min_{i\neq j}|u_i-u_j|>0$. By \eqref{eq:recursion-sec2} and
\eqref{eq:submultiplicative-sec2}, the off-diagonal part satisfies an estimate of the form
\[
\|H_{r+1}^{\mathrm{off}}\|
\le
\Delta^{-1}\Bigl(\frac{r}{|h|}+\|T\|+\|\delta T\|\Bigr)\|H_r\|.
\]
For the diagonal part, the inverses of the operators in \eqref{eq:diagonal-sec2} are bounded by $C_0/(r+1)$ for a
suitable constant $C_0>0$, since the spectrum of $\ad(e_{ii})$ is finite. It follows that there exists a constant
$C_1>0$ such that
\[
\|H_{r+1}\|\le C_1(r+1)\|H_r\|,
\qquad r\ge0.
\]
The factorial estimate follows by induction from $H_0=I$.
\end{proof}
\end{comment}

\subsection{Canonical solutions with prescribed asymptotics via Borel--Laplace summation}\label{sec:2.3}

Since $\mathcal A_V$ is now a Banach algebra, the standard Borel--Laplace summation theory (c.f. \cite{Balser, LR}) applies to \eqref{eq:MonoSingularSys}. 

\begin{defi}\label{def:anti-stokes-sec2}
The anti-Stokes directions of \eqref{eq:MonoSingularSys} are the directions along which
$e^{h(u_i-u_j)z}$ decays most rapidly as $z\to\infty$ for some $u_i\neq u_j$. Denote by ${\rm{aSR}}(hu)$ the set of
anti-Stokes directions. Then
\begin{equation}\label{eq:anti-stokes-sec2}
{\rm{aSR}}(hu)=\{-{\rm{Arg}}(u_i-u_j)+2k\pi:\ k\in\mathbb Z,\ i\neq j\}.
\end{equation}
A direction $d$ is called an admissible direction if $d\notin {\rm{aSR}}(hu)$.
\end{defi}

In this paper we denote a direction/ray by its argument. Let us choose an initial anti-Stokes direction
$\tau_0$ and then arrange the anti-Stokes directions into a strictly monotonically increasing sequence
\[
\cdots<\tau_{-1}<\tau_0<\tau_1<\cdots.
\]

\begin{defi}\label{def:stokes-sector-sec2}
For any admissible direction $d\in(\tau_j,\tau_{j+1})$, the Stokes sector $\Sect_d$ is defined as
\begin{equation}\label{eq:stokes-sector-sec2}
\Sect_d
:=
\Bigl\{
z\in\widetilde{\mathbb C\setminus\{0\}}:\ 
\arg(z)\in\Bigl(\tau_j-\frac{\pi}{2},\,\tau_{j+1}+\frac{\pi}{2}\Bigr)
\Bigr\}.
\end{equation}
\end{defi}

\begin{pro}\label{prop:canonical-solution-sec2}
Given any $h\notin\mathbb Q$, fixed $u\in\mathfrak h_{\mathrm{reg}}$ and an admissible direction $d$, there exists a
unique $\mathcal A_V$-valued holomorphic fundamental solution $F_d(z)$ of \eqref{eq:MonoSingularSys} on
$\widetilde{\mathbb C\setminus\{0\}}$ with the asymptotics
\begin{equation}\label{eq:canonical-asymptotics-sec2}
F_d(z)\cdot z^{-h\delta T}e^{-huz}\sim I,
\qquad
z\to\infty \ \text{within }\Sect_d.
\end{equation}
Since $F_d(z)$ is the unique solution with the prescribed asymptotics
\eqref{eq:canonical-asymptotics-sec2}, it is called the canonical solution (with respect to the sector $\Sect_d$) at $z=\infty$.
\end{pro}

\begin{proof}
Since the proof is standard, we only outline the construction of the canonical solution $F_d(z)$. First recall that
the formal Borel transform of a formal power series
\[
\widehat f(z)=I+a_1z^{-1}+a_2z^{-2}+\cdots
\]
at $z=\infty$ is defined by
\begin{equation}\label{eq:borel-transform-sec2}
\mathcal B(\widehat f)(\xi)
:=
\sum_{p=1}^{\infty}\frac{a_p}{(p-1)!}\xi^{p-1}.
\end{equation}
The Laplace transform of a holomorphic $\mathcal A_V$-valued function $g(\xi)$ along the direction $-d$ is defined by
\begin{equation}\label{eq:laplace-transform-sec2}
(\mathcal L_d g)(z)
=
I+\int_{0}^{+\infty\cdot e^{-\mathrm{i}d}} e^{-z\xi}g(\xi)\,d\xi,
\end{equation}
where the line of integration is the ray with argument $-d$.

By the standard estimates in the Banach-valued Borel--Laplace theory, the Borel transform
$\mathcal B(\widehat H)$ extends analytically in a sector around the direction $-d$ and has at most exponential
growth there. Hence the Laplace transform
\[
H_d(z):=\mathcal L_d\bigl(\mathcal B(\widehat H)\bigr)(z)
\]
is well defined for $z$ in a sector of opening $\pi$ centered at the direction $d$. It is also known that
$H_d(z)$ is asymptotic to the formal power series $\widehat H(z)$. 

By the compatibility of the Borel--Laplace transform and the derivative with respect to $z$, the holomorphic
function $H_d(z)$ satisfies the same equation as the formal power series $\widehat H(z)$. As a consequence, the
function
\begin{equation}\label{eq:Hd-to-Fd-sec2}
F_d(z):=H_d(z)e^{huz}z^{h\delta T}
\end{equation}
solves \eqref{eq:MonoSingularSys}. By the asymptotic property of $H_d(z)$, the solution $F_d(z)$ satisfies
\eqref{eq:MonoSingularSys}. Finally, by analytic continuation, it extends to
$\widetilde{\mathbb C^\times}$, still with the prescribed asymptotics on $\Sect_d$. The uniqueness follows from
Watson's lemma (c.f.\cite{Watson} or \cite[Chap.4, Prop.11]{Balser}).
\end{proof}

\subsection{Stokes (super)matrices}\label{sec:2.4}

We next define the Stokes matrices of \eqref{eq:MonoSingularSys}. By Proposition~\ref{prop:canonical-solution-sec2}, the function $F_d(z)\,z^{-h\delta T}e^{-huz}$ is asymptotic to the identity as $z\to\infty$ within $\Sect_d$. Since $\mathcal A_V$ is a Banach algebra, this implies that $F_d(z)$ is invertible. Moreover, the factor $e^{huz}z^{h\delta T}$ is an invertible even diagonal element of $\mathcal A_V$. Hence one can take the transition matrices between the canonical solutions at different sectors. 

\begin{defi}\label{def:stokes-matrix-sec2}
For fixed $u\in\mathfrak{h}_{\rm reg}$, $h\notin \mathbb{Q}$ and any admissible direction $d\notin {\rm{aSR}}(hu)$, the associated Stokes matrices of the equation \eqref{eq:MonoSingularSys} are the
constant elements $S^{\pm}_{d}(u)\in\mathcal A_V$ determined by
\begin{equation}\label{eq:stokes-definition-sec2}
F_{d}(z)=F_{d\pm\pi}(z)\cdot S^{\pm}_{d}(u).
\end{equation}
\end{defi}

\begin{rmk}\label{rmk:stokes mat}
According to Proposition~\ref{prop:formal-solution-sec2},  ${H}_{r,ij}$ has the parity $\bar{i}+\bar{j}$. Since the resummation process preserves the $\mathbb{Z}_2$-grading of the basis, the $(i,j)$-entry of the sectorial solution $F_{d}(z)$ likewise inherits the parity $\bar{i}+\bar{j}$. Consequently, the $(i,j)$-entry of the Stokes matrix $S^{\pm}_d(u)$ is also of parity $\bar{i}+\bar{j}$. 
For computational convenience, in practice we can identify these tensor products of the form $\sum_{i,j} A_{ij}\otimes E_{ij}$ (where $|A_{ij}|=\bar i+\bar j$) with a matrix $(\widetilde{A}_{ij})$, where the entries are defined by $\widetilde{A}_{ij} = A_{ij} (-1)^{\bar{i}(\bar{j}+\bar{1})}$. This assignment allows us to perform usual matrix multiplication on $(\widetilde{A}_{ij})$. The final result can then be pulled back to the tensor product form simply by reintroducing the corresponding signs. Under this identification, the Stokes matrix $S^{\pm}_d(u)$ can be viewed as a supermatrix, which is even due to its grading structure (for supermatrix, see e.g. \cite{Carmeli2011, Manin1997}). It should be emphasized that we adopt this identification exclusively for the explicit computations in Section \ref{sec:2.5}; throughout the remainder of this paper, the standard tensor product notation is retained.
\end{rmk}

To describe the triangularity of $S^{\pm}_d(u)$, let $\sigma_d$ be the permutation of $\{1,\dots,m+n\}$ corresponding
to the dominance ordering of $\{e^{hu_1z},\dots,e^{hu_{m+n}z}\}$
along the direction $d+\frac{\pi}{2}$. That is,
\[
\sigma_d(i)<\sigma_d(j)
\quad\Longleftrightarrow\quad
e^{h(u_i-u_j)z}\to 0
\qquad\text{as } z\to\infty \text{ along } d+\frac{\pi}{2}.
\]

\begin{pro}\label{prop:triangularity-sec2}
Write
\[
S^{\pm}_d(u)=\sum_{i,j=1}^{m+n}(S^{\pm}_d)_{ij}\otimes E_{ij}.
\] Then the matrix $S_d^+(u)$ is upper triangular with respect to the ordering $\sigma_d$, while $S_d^-(u)$ is lower
triangular with respect to the same ordering. More precisely,
\[
(S_d^{\pm}(u))_{ij}=0
\qquad\text{whenever }\sigma_d(i)\gtrless\sigma_d(j).
\]
Furthermore, their diagonal entries are
\begin{equation}\label{eq:stokes-diagonal-sec2}
(S_d^+(u))_{ii}=(S_d^-(u))_{ii}=1\in\End(V),
\qquad i=1,\dots,m+n.
\end{equation}

\end{pro}

\begin{proof}
We only prove $S_d^+(u)$, $S_d^-(u)$ is similar. By Definition~\ref{def:stokes-matrix-sec2} and Proposition~\ref{prop:canonical-solution-sec2}, one has
\begin{equation*}
e^{huz}z^{h\delta T}\,S_d^{\pm}(u)\,z^{-h\delta T}e^{-huz}
=
\bigl(F_{d\pm\pi}(z)z^{-h\delta T}e^{-huz}\bigr)^{-1}\bigl(F_d(z)z^{-h\delta T}e^{-huz}\bigr)\to I,
\end{equation*}
as $z\to\infty$ within $\Sect_d\cap \Sect_{d\pm\pi}$. Therefore, for $i\neq j$, the $(i,j)$-entry of $S_d^{\pm}(u)$ can be nonzero only if $e^{h(u_i-u_j)z}\to 0$ on $\Sect_d\cap \Sect_{d\pm\pi}$. By the definition of $\sigma_d$, this implies that
\[
(S_d^{\pm}(u))_{ij}=0
\qquad\text{whenever }\sigma_d(i)\gtrless\sigma_d(j).
\]
For $i=j$, the same limits give
\[
z^{h\delta_i}(S_d^\pm(u))_{ii}z^{-h\delta_i}\to 1,
\qquad
\delta_i:=(-1)^{\bar i}e_{ii}.
\]
Since $(S_d^\pm(u))_{ii}$ is independent of $z$, it follows that $(S_d^\pm(u))_{ii}=1$. This proves \eqref{eq:stokes-diagonal-sec2}.
\end{proof}

\begin{comment}
\begin{rmk}\label{rmk:ordering-sec2}
\rednote{remark modifiy: delete or state ordering is equivalent to super permutation; always choose approriate $d$.}
For a diagonal element $u$ with pairwise distinct entries, we can choose a permutation matrix $J$ and perform the
super-conjugation
\[
F(z)\longmapsto (1\otimes J)\cdot_s F(z),
\]
which transforms \eqref{eq:main-system-sec2} into a system of the same form with the diagonal term reordered.
Accordingly, the transformed system may be arranged so that its Stokes matrices satisfy the triangularity statement
of Proposition~\ref{prop:triangularity-sec2} with respect to the standard order $1<\cdots<m+n$. Because this is a
super-conjugation with respect to the product $\cdot_s$, additional parity-dependent signs may appear in the
transformed coefficient matrix. In the sequel, we shall always assume that $u$ has been arranged in this way, so
that, for the chosen admissible direction, $S_d^+(u)$ is upper triangular and $S_d^-(u)$ is lower triangular in the
standard order.
\end{rmk}
\end{comment}

\begin{lem}\label{lem:stokes-periodicity-sec2}
For any admissible direction $d$, one has
\[
S_{d+2\pi}^\pm(u)=e^{-2\pi \mathi h\,\delta T}S_d^\pm(u)e^{2\pi \mathi h\,\delta T}.
\]
\end{lem}

\begin{proof}
Since
\[
F_{d+2\pi}(ze^{2\pi \mathi})=F_d(z)e^{2\pi \mathi h\,\delta T},
\]
the identity follows immediately from Definition~\ref{def:stokes-matrix-sec2}.
\end{proof}

This lemma shows that it suffices to consider a single pair of $S^{\pm}_d(u)$, from which the Stokes information in all other directions can be retrieved. For the algebraic construction in the next section, we shall use the following renormalized matrices.

\begin{defi}\label{def:renorm-Stokes}
    Choose an appropriate admissible direction $d$ such that $\sigma_d(i)>\sigma_d(j)$ whenever $i>j$, i.e. $S_d^+(u)\in\mathcal{A}_V$ is upper triangular, define the renormalized Stokes (super)matrices $S_{h\pm}(u)$ to be:
    \begin{equation}
    S_{h+}(u)=S_d^+(u)\cdot e^{\pi \mathi h\delta T};\qquad 
    S_{h-}(u)=e^{-\pi \mathi h\delta T}\cdot \big(S_d^-(u)\big)^{-1}.
    \end{equation}
\end{defi}

\subsection{Example: Stokes matrices associated to the natural representation and the standard super \texorpdfstring{$R$}{R}-matrices of quantum supergroup}\label{sec:2.5}
In this section, we compute  the Stokes matrices of the system \eqref{eq:MonoSingularSys} associated to the dual representation of the natural representation 
$V\cong \mathbb C^{m|n}$. Via the representation 
\[U(\mathfrak{gl}_{m|n})\rightarrow \End(\mathbb{C}^{m|n})~;~\rho(e_{ij})=-(-1)^{\bar i(\bar j+1)}E_{ji},\] the system can be written as
\begin{equation}\label{eq:rank-twosystemsec2}
\frac{\dd F}{\dd z}= h\Bigl(U+\frac{P}{z}\Bigr) F,
\end{equation}
where
\[
U:=\sum_{i=1}^{m+n}1\otimes u_iE_{ii},
\qquad 
P:=-\sum_{k,l=1}^{m+n}(-1)^{\bar l}E_{kl}\otimes E_{lk}.
\]

 Using the viewpoint in Remark~\ref{rmk:stokes mat} and choosing the natural basis $\{e_1,\dots,e_{m+n}\}$ of $\mathbb{C}^{m|n}$, we can identify the above system with a matrix system under usual matrix multiplication. Here, the standard matrix units $E_{ij}$ have $1$ at the $(i,j)$-entry and $0$ elsewhere, and the matrices $U$ and $P$ are given by
\begin{align}\label{eq:Psys in mat form}
U = \mathrm{diag}(u_1\cdot \mathrm{Id},\dots,u_{m+n}\cdot \mathrm{Id}), \quad P = (P_{ij}) = \big(-(-1)^{\bar{i}\cdot\bar{j}}E_{ij}\big).
\end{align}

%The constructions of the previous subsections apply to \eqref{eq:rank-two-system-sec2}. Then the main result for this subsection is 
\begin{thm}\label{thm:super-r-matrix}
The inverse of the (renormalized) Stokes matrix $S_{h+}(U)^{-1}$ of system \eqref{eq:rank-twosystemsec2} equals to \(R \in \End(\mathbb{C}^{m|n})\otimes \End(\mathbb{C}^{m|n})\), where
\begin{equation}\label{eq:standard-r-matrix}
R
=
\sum_{i=1}^{m+n} q^{(-1)^{\bar i}} E_{ii}\otimes E_{ii}
+
\sum_{i\neq j} E_{ii}\otimes E_{jj}
+
(q-q^{-1})\sum_{i<j}(-1)^{\bar i} E_{ji}\otimes E_{ij},
\qquad q=e^{\pi i h},
\end{equation}
is the standard super $R$-matrix (c.f.\cite{PerkSchultz1981,Zhang2016}) of $U_q(\mathfrak{gl}(m|n))$.
\end{thm}

\begin{proof}
We begin with $m+n=2$, where $(m,n)$ can be $(2,0)$, $(1,1)$, or $(0,2)$.  Under the matrix system \eqref{eq:Psys in mat form}, the corresponding $P$ are 
\begin{equation*}
    P^{2|0}= -\begin{pmatrix}
        1&0&0&0\\
        0&0&1&0\\
        0&1&0&0\\
        0&0&0&1
    \end{pmatrix},\quad
    P^{1|1}= -\begin{pmatrix}
        1&0&0&0\\
        0&0&1&0\\
        0&1&0&0\\
        0&0&0&-1
    \end{pmatrix},\quad 
    P^{0|2}= -\begin{pmatrix}
        -1&0&0&0\\
        0&0&-1&0\\
        0&-1&0&0\\
        0&0&0&-1
    \end{pmatrix}.
\end{equation*}
Thus we only need to compute the Stokes matrices of  the following $2 \times 2$ system:
\begin{align}\label{eq:ordinary-ranktwo-sec2}
&\frac{\dd F}{\dd z}=h\left(
\begin{pmatrix}
u_2&0\\[2pt]
0&u_3
\end{pmatrix}
+\frac{1}{z}
\begin{pmatrix}
0&\sigma\\[2pt]
\sigma&0
\end{pmatrix}
\right) F,\quad \sigma=\pm1.
\end{align}
Choose a direction $d$ for which the Stokes matrix $S_d^+(\sigma)$ of the system \eqref{eq:ordinary-ranktwo-sec2} is upper triangular. Applying the formula of Stokes matrices of a general system of rank $2$ in \cite[Proposition 8]{BJL}, see also \cite{Xu2019}, and using the Euler's reflection formula $\Gamma(x)\Gamma(1-x)=\frac{\pi}{\sin \pi x}$, we get
\begin{equation}\label{eq:BJL stokes}
S_d^+(\sigma)=
\begin{pmatrix}
1&e^{\pi \mathi\sigma h}-e^{-\pi \mathi\sigma h}\\[2pt]
0&1
\end{pmatrix}.
\end{equation}
Letting $q=e^{\pi \mathi h}$ and substituting \eqref{eq:BJL stokes} into the renormalized Stokes matrices $S_{h+}(U)^{-1}=q^{-h\delta P}\cdot \left(S_{d+}(\sigma)\right)^{-1}$, we have
\begin{align*}
    (S_{h+}^{2|0})^{\!-1}\!=\!
    \begin{pmatrix}
        q&0&0&0\\
        0&1&q\!-\!q^{\!-1}&0\\
        0&0&1&0\\
        0&0&0&q 
    \end{pmatrix},\quad
    (S_{h+}^{1|1})^{\!-1}\!=\!
    \begin{pmatrix}
        q&0&0&0\\
        0&1&q\!-\!q^{\!-1}&0\\
        0&0&1&0\\
        0&0&0&\!\!\!q^{\!-1} 
    \end{pmatrix},\quad
  (S_{h+}^{0|2})^{\!-1}\!=\!
   \begin{pmatrix}
        q^{-1}\!\!\!&0&0&0\\
        0&1&q^{\!-1}\!-\!q&0\\
        0&0&1&0\\
        0&0&0&\!\!\!q^{-1}
    \end{pmatrix}.
\end{align*}

\begin{comment}
\begin{align*}
    (S_{h+}^{2|0})^{-1}=
    {\scriptsize\begin{pmatrix}
        q&0&0&0\\
        0&1&q-q^{-1}&0\\
        0&0&1&0\\
        0&0&0&q 
    \end{pmatrix}},\quad
    (S_{h+}^{1|1})^{-1}=
    {\scriptsize\begin{pmatrix}
        q&0&0&0\\
        0&1&q-q^{-1}&0\\
        0&0&1&0\\
        0&0&0&q^{-1} 
    \end{pmatrix}},\quad
  (S_{h+}^{0|2})^{-1}=
   {\scriptsize \begin{pmatrix}
        q^{-1}&0&0&0\\
        0&1&q^{-1}-q&0\\
        0&0&1&0\\
        0&0&0&q^{-1}
    \end{pmatrix}}.
\end{align*}
\end{comment}

By multiplying the entries of each block matrix by $(-1)^{\bar{i}(\bar{j}+\bar{1})}$ to return to the tensor product form, we can verify that
 \[
 S_{h+}(U)^{-1}=  R,\quad \text{for}\ m+n=2.
\]

Now let us consider general  case for $m+n\geq 3$. Notice that for any $e_k,\ e_l,\ k<l$,  $P(e_k\otimes e_l)=-(-1)^{\bar k \bar l} e_l\otimes e_k$ and $P(e_k\otimes e_k)=-(-1)^{\bar k} e_{kk}\otimes e_{kk}$. Therefore, the system \eqref{eq:rank-twosystemsec2} can be restricted to any two-dimensional subspace $\mathrm{span}\{e_k\otimes e_l, e_l\otimes e_k\}$ or one-dimensional subspace $\mathrm{span}\{e_k\otimes e_k\}$. Consequently, it suffices to show that on these invariant subspaces, the action of $R$ coincides with the inverse of the action of the renormalized Stokes matrix of the system. 

By definition, the assertion holds trivially on each one-dimensional subspace $\mathrm{span}\{e_k\otimes e_k\}$, where $R = q^{(-1)^{\bar{k}}}\mathrm{Id}$ and $S_{d+}=\mathrm{Id}$, yielding $R = S_{h+}(U)^{-1}$. 
On the two-dimensional subspaces $\mathrm{span}\{e_k\otimes e_l, e_l\otimes e_k\}$ for $k<l$, the $R$-matrix restricted to this subspace reads
\[
R = E_{kk}\otimes E_{ll} + E_{ll}\otimes E_{kk} + (q-q^{-1}) (-1)^{\bar{k}} E_{lk}\otimes E_{kl},
\]
while the system \eqref{eq:rank-twosystemsec2} takes the form $U = \mathrm{diag}(u_k, u_l)$ and $P = -(-1)^{\bar{l}}E_{kl}\otimes E_{lk} - (-1)^{\bar{k}}E_{lk}\otimes E_{kl}$. This coincides precisely with the previously verified $m+n=2$ case, thereby completing the proof that $R = S_{h+}(U)^{-1}$ globally.
\end{proof}

%% file: final_sections/3RLLrelation.tex
\section{Stokes matrices satisfy the RLL relation}\label{sec:3}

Let \(R \in \End(\mathbb{C}^{m|n})\otimes \End(\mathbb{C}^{m|n})\) be the standard \(R\)-matrix of \(U_q(\mathfrak{gl}(m|n))\) as in Theorem \ref{thm:super-r-matrix}, the main result of this section is the following.

\begin{thm}\label{thm:RLL-relation}
For any \(h\notin \mathbb{Q}\) and \(u\in \mathfrak h_{\mathrm{reg}}\), the Stokes matrices \(S_{h\pm}(u)\) of \eqref{eq:MonoSingularSys} satisfy
\begin{subequations}\label{eq:rll-relation}
\begin{align}\label{eq:rll-plus}
R^{12} S_{h+}^{(1)}(u) S_{h+}^{(2)}(u)=S_{h+}^{(2)}(u) S_{h+}^{(1)}(u) R^{12},\\
\label{eq:rll-minus}
R^{12} S_{h-}^{(1)}(u) S_{h-}^{(2)}(u)=S_{h-}^{(2)}(u) S_{h-}^{(1)}(u) R^{12},\\
\label{eq:rll-mixed}
R^{12} S_{h+}^{(1)}(u) S_{h-}^{(2)}(u)=S_{h-}^{(2)}(u) S_{h+}^{(1)}(u) R^{12}.
\end{align}
\end{subequations}
Here we take the convention
\[S^{(1)}_{h\pm}(u):= \sum_{i,j}S_{h\pm}(u)_{ij}\otimes E_{ij}\otimes 1, \ S^{(2)}_{h\pm}(u):= \sum_{i,j}S_{h\pm}(u)_{ij}\otimes 1\otimes E_{ij}, \text{ and } \ R^{12}:=1\otimes R\] 
as elements in $\End(V)\otimes \End(\mathbb{C}^{m|n})\otimes \End(\mathbb{C}^{m|n})$.
\end{thm}

The proof is based on the super irregular KZ equations associated with \(\mathfrak{gl}_{m|n}\).  
More precisely, in Section~\ref{sec:3.1} we realize the coefficient matrices \(T^{(1)},T^{(2)}\) and \(P\) from the Casimir element of \(\mathfrak{gl}_{m|n}\). In Section \ref{sec:3.2}, we construct the formal solution $\widehat{Y}(z,t)$ in Proposition \ref{thm:FormalSolY} of the above system and introduce three factorization formulas in Lemma \ref{lem:DecompY}. In Section \ref{sec:3.3}, we prove the holomorphic properties of decomposed functions in Theorem \ref{thm:Y-decomp}, and in Section \ref{sec:3.4} we use ismonodromy arguments to trivialize some Stokes rays of $\widehat{Y}(z,t)$, and in Theorem \ref{thm:Y-on-chamber} we construct holomorphic solutions on different $t-$chambers (Definition \ref{def:Six-solutions}). Finally in Section \ref{sec:3.5}, we derive explicit connection formulae \eqref{eq:step-II2}, \eqref{eq:step-I1}-\eqref{eq:step-II3} between the solutions on different $t-$chambers, using which we prove Theorem \ref{thm:RLL-relation}.

\input{final_sections/section3/Section3.1}

\input{final_sections/section3/Section3.2}

\input{final_sections/section3/Section3.3}

\input{final_sections/section3/Section3.4}

\input{final_sections/section3/Section3.5}

%which cut the $t-$plane into several connected components, as figure \ref{pic:t-plane} shows. Fix $t\notin \mathrm{aSR}_t(hu;d)$ and perform the Borel-Laplace transformation with respect to $z$, it is well known that the resummation process is a differential algebra homomorphism between formal power series and holomorphic functions, which are related by asymptotic expansion in sectorial regions. More precisely, we have:

%\input{final_pictures/pic_minipage}
%\input{final_pictures/pic_RLL}
%\input{final_pictures/pic_t-plane}
%\input{final_pictures/pic_RLL}

%% file: final_sections/section3/Section3.1.tex
\subsection{The super irregular KZ equations for \texorpdfstring{$\mathfrak{gl}_{m|n}$}{gl(m|n)}}
\label{sec:3.1}

We begin with the super irregular KZ equations associated with $\mathfrak{gl}_{m|n}$. Following \cite{Geer2007}, we take the Casimir element
\begin{equation*}\label{eq:normalized-casimir-glmn}
\Omega:=-\sum_{i,j=1}^{m+n}(-1)^{\bar j}e_{ij}\otimes e_{ji}.
\end{equation*}
For $\mathfrak{gl}_{m|n}$-modules $W_0,\dots,W_N$, the corresponding super irregular KZ system for an $
{\rm End}(W_0\otimes W_1\otimes \cdots \otimes W_N)
$ valued function $\Phi(z_1,...,z_N)$
is
\begin{equation}\label{eq:general-super-irregular-kz}
\frac{1}{h}\frac{\partial \Psi}{\partial z_i}
=
\left(
u^{(i)}+\frac{\Omega_{0i}}{z_i}+\sum_{j\neq i}\frac{\Omega_{ij}}{z_i-z_j}
\right)\Psi,
\qquad i=1,\dots,N,
\end{equation}
where \(u^{(i)}\) denotes the action of $u\in\mathfrak{h}_{\rm reg}$ on the \(i\)-th tensor factor, and \(\Omega_{ab}\) denotes the action of \(\Omega\) on the \(a\)-th and \(b\)-th tensor factors. Thus the irregular part is governed by \(u^{(i)}\), while the logarithmic interactions are encoded by the Casimir element.

\begin{pro}\label{pro:general-super-irregular-kz-flat}
The system \eqref{eq:general-super-irregular-kz} is compatible.
\end{pro}

\begin{proof}
Let
\[
A_i:=u^{(i)}+\frac{\Omega_{0i}}{z_i}+\sum_{j\neq i}\frac{\Omega_{ij}}{z_i-z_j}.
\]
Since $\Omega$ is even, super-symmetric, and invariant, it satisfies
\[
|\Omega|=0,\qquad \tau(\Omega)=\Omega,\qquad [x\otimes 1+1\otimes x,\Omega]=0
\quad \text{for all }x\in \mathfrak{gl}(m|n),
\]
where $\tau(a\otimes b)=(-1)^{|a||b|}b\otimes a$ is the super-flip. These identities imply the relations
\[
[\Omega_{ab},\Omega_{cd}]=0,\qquad
[\Omega_{ab},\Omega_{ac}+\Omega_{bc}]=0,\qquad
[u^{(a)}+u^{(b)},\Omega_{ab}]=0,
\]
whenever the indices are pairwise distinct. Hence \(\partial_{z_j}A_i=\partial_{z_i}A_j\) and \([A_i,A_j]=0\), the system is compatible.
\end{proof}

We now specialize to the case needed for the proof of the RLL relation. Take \(N=2\), let \(W_0=W\) be any finite-dimensional \(\mathfrak{gl}_{m|n}\)-module, and let the second and third tensor factors be \(\mathbb C^{m|n}\) equipped with the dual natural representation
\begin{equation*}\label{eq:dual-natural-representation}
\rho(e_{ij})=-(-1)^{\bar i(\bar j+1)}E_{ji}.
\end{equation*}
Writing the resulting system in matrix form with respect to the last two tensor factors, we obtain a system for an \(\End(W)\otimes \End(\mathbb C^{m|n})\otimes \End(\mathbb C^{m|n})\)-valued function \(Y(z_1,z_2)\), in which
\[
\Omega_{01}=T^{(1)},\qquad \Omega_{02}=T^{(2)},\qquad \Omega_{12}=P,
\]
where
\begin{equation*}\label{eq:specialized-kz-operators}
T^{(1)}=\sum_{k,l}(-1)^{\bar k\bar l}e_{kl}\otimes E_{kl}\otimes 1,\qquad
T^{(2)}=\sum_{k,l}(-1)^{\bar k\bar l}e_{kl}\otimes 1\otimes E_{kl},\qquad
P=-\sum_{k,l}(-1)^{\bar l}1\otimes E_{kl}\otimes E_{lk}.
\end{equation*}
Thus the case \(N=2\) of \eqref{eq:general-super-irregular-kz} becomes
\begin{align}
\frac{1}{h}\frac{\partial Y}{\partial z_1}
&=
\left(
u^{(1)}+\frac{T^{(1)}}{z_1}+\frac{P}{z_1-z_2}
\right)\cdot Y,
\label{eq:super-irregular-kz-z1}
\\
\frac{1}{h}\frac{\partial Y}{\partial z_2}
&=
\left(
u^{(2)}+\frac{T^{(2)}}{z_2}+\frac{P}{z_2-z_1}
\right)\cdot Y,
\label{eq:super-irregular-kz-z2}
\end{align}
where
\[
u^{(1)}=\sum_i 1\otimes u_iE_{ii}\otimes 1,\qquad
u^{(2)}=\sum_i 1\otimes 1\otimes u_iE_{ii}.
\]
%As a specialization of the compatible system \eqref{eq:general-super-irregular-kz}, the system \eqref{eq:super-irregular-kz-z1}--\eqref{eq:super-irregular-kz-z2} is compatible as well.

For the asymptotic analysis below, it is convenient to separate the overall irregular scale from the relative position of the two points by introducing$z_1=z,\,z_2=zt$. In these coordinates, \eqref{eq:super-irregular-kz-z1}--\eqref{eq:super-irregular-kz-z2} becomes
\begin{align}
\frac{1}{h}\frac{\partial Y}{\partial z}
&=
\left(
u^{(1)}+t\,u^{(2)}+\frac{T^{(1)}+T^{(2)}+P}{z}
\right)\cdot Y,
\label{eq:super-irregular-kz-z}
\\
\frac{1}{h}\frac{\partial Y}{\partial t}
&=
\left(
z\,u^{(2)}+\frac{T^{(2)}}{t}+\frac{P}{t-1}
\right)\cdot Y.
\label{eq:super-irregular-kz-t}
\end{align}
%In what follows, we consider only super systems, and we drop the subscript "s" in $\cdot_s$ for simplicity, when no confusion arises.
\begin{rmk}\label{rmk:casimir-normalization}
Compared with the standard convention in the literature, the Casimir element in \eqref{eq:normalized-casimir-glmn} differs by an overall minus sign. We adopt this normalization so that its specialization gives exactly the operators \(T\) and \(P\).
\end{rmk}

%% file: final_sections/section3/Section3.2.tex
\subsection{Formal solution of KZ equation in \texorpdfstring{$z,t$}{z,t} variables and its decomposition}\label{sec:3.2}

\begin{pro}\label{thm:FormalSolY}
For any fixed $h\notin \mathbb{Q}$ and $u\in \hreg$, the system of equations \eqref{eq:super-irregular-kz-z} and \eqref{eq:super-irregular-kz-t} has a unique formal fundamental solution taking the form 
\begin{align}\label{eq:FormalSolY}
\widehat{Y}(z,t)=\widehat{Q}(z;t) e^{ hz(u^{(1)}+tu^{(2)})}z^{h(\delta T^{(1)}+\delta T^{(2)}+\delta P)}t^{h\delta T^{(2)}}(t-1)^{h\delta P}, 
\end{align}
where $\widehat{Q}=1+Q_1z^{-1}+Q_2z^{-2}+\cdot\cdot\cdot$ is a formal power series in $z^{-1}$, and each coefficient $Q_r(t)\in {\rm End}(V)\otimes {\rm End}(\IC^{m|n})\otimes {\rm End}(\IC^{m|n})$ is a rational function of $t$ having poles at $t=0,1,\infty$.
\end{pro}

\begin{proof}
%\rednote{to be modified, see finalproofsFormalSolY}
%\input{final_proofs/3.2-FormalSol-Y}

%\rednote{HERE SEPARATION}

\input{final_proofs/3.2-Formal-Sol-Y}
\end{proof}

In the rest of this subsection, we  present three decompositions of $\widehat{Y}$. As will become clear, these decompositions correspond to the neighborhoods of the singular points $t=0,\infty,1$, respectively. In each case, the former factor is holomorphic near the corresponding point, while the latter captures the local singular behavior, see Theorem \ref{thm:Y-decomp}.

\begin{pro}\label{lem:DecompY}
    \input{final_bigThms/3.2-Lem-DecompY}
\end{pro}

\begin{proof}
\input{final_bigThms/3.2-Proof-Lem-DecompY}
\end{proof}

\begin{rmk}\label{rmk:multivalue_1-t}
    In the first factorization of Proposition \ref{lem:DecompY}, the multi-valued function $(1-t)^{h\delta P}$ appearing in 
$_1\widehat{K}\!=\!(1\!-\!t)^{h\delta P}\!+\!O(z^{-1})$ is taken to have the expansion $\mathrm{Id}+O(t)$. That is the branch $\log(1-t)|_{t=0}=0$ is chosen, while the branch of $(t-1)^{h\delta P}$ in the singular part of $\widehat{Y}(z,t)$ is not fixed. Thus, throughout this paper, the factor $(\frac{t-1}{1-t})^{h\delta P}$ is computed by
\begin{equation*}
    \left(\frac{t-1}{1-t}\right)^{h\delta P}=
    \begin{cases}
    e^{(2k+1)\pi\mathi h\delta P},\quad  \text{ if }\arg(t-1)\in(2k\pi,(2k+1)\pi);\\
    e^{(2k-1)\pi\mathi h\delta P},\quad \text{ if }\arg(t-1)\in((2k-1)\pi, 2k\pi).
    \end{cases}
    \end{equation*}
\end{rmk}

\begin{comment}
We now consider the $z,t$-systems satisfied by the second factors in the above decompositions.
\begin{pro}\label{thm:reduced-equation}
In each of the three cases of Lemma \ref{lem:DecompY}, the $z,t$-system of the second component reduces to a single equation via a suitable change of variables. More precisely, we have
\begin{enumerate}
    \item[(a).]($t=0$) 
    Under $\sigma_1=z,\,\eta=zt$, $\widehat{F}^{(2)}$ in \eqref{eq:zt-funcW1F2} satisfy $\partial \widehat{F}^{(2)}/\partial \sigma_1=0$ and
    \begin{equation}\label{eq:monoF2}
        \frac{\partial \widehat{F}^{(2)}}{\partial \eta}=\left(u^{(2)}+\frac{T^{(2)}}{\eta}\right)\widehat{F}^{(2)}.
    \end{equation}

    \item[(b).]($t=\infty$) 
    Under $\sigma_1=zt,\,\eta=z$, $\widehat{F}^{(1)}$ in \eqref{eq:zt-funcW2F1} satisfy $\partial \widehat{F}^{(1)}/\partial \sigma_1=0$ and
    \begin{equation}\label{eq:monoF1}
        \frac{\partial \widehat{F}^{(1)}}{\partial \eta}=\left(u^{(1)}+\frac{T^{(1)}}{\eta}\right)\widehat{F}^{(1)}.
    \end{equation}

    \item[(c).]($t=1$) 
    Under $\sigma_1=z,\,\eta=zt-z$, $\widehat{X}$ in \eqref{eq:zt-funcUX} satisfy $\partial \widehat{X}/\partial \sigma_1=0$ and
    \begin{equation}\label{eq:monoX}
        \frac{\partial \widehat{X}}{\partial \eta}=\left(u^{(2)}+\frac{P}{\eta}\right)\widehat{X}.
    \end{equation}
\end{enumerate}
\end{pro}

\begin{proof}
    Direct computation by chain rule.
\end{proof}
\end{comment}

%\input{final_bigThms/3.2-Rem-tEquations}

%% file: final_proofs/3.2-Formal-Sol-Y.tex
Given $u\in\hreg$, define the domain
\begin{equation}\label{eq:t-domain}
D_{u}:=\left\{t\in \IC~|~t\ne \frac{u_j-u_i}{u_k-u_l}, \text{ for any } i,j, k,l\right\}.\end{equation} 
We will first construct the unique formal solution $\widehat{Y}(z;t)$ of \eqref{eq:super-irregular-kz-z} for $t\in D_u$ as in Proposition \ref{prop:formal-solution-sec2}, then prove that this function satisfies \eqref{eq:super-irregular-kz-t} and extend the definition to all $t$.

\smallskip
\noindent\textit{\textbf{Step 1. Recurrence relation}}

Inserting \eqref{eq:FormalSolY} into \eqref{eq:super-irregular-kz-z} and comparing the coefficients of $z^{-r-1}$, we derive the equation for $\widehat{Q}$,
\begin{align}\label{eq:Qhat}
\frac{1}{h}\frac{\dd \hat{Q}}{\dd z}+\hat{Q} \left(u^{(1)}+t u^{(2)} +
\frac{\delta T^{(1)}+\delta T^{(2)}+\delta P}{z}\right)
=\left(u^{(1)}+t u^{(2)}+
\frac{T^{(1)}+T^{(2)}+P}{z}\right) \hat{Q},
\end{align}
and the recurrence relation

\begin{align}\label{eq:Qm}
[Q_{r+1}(t), u^{(1)}+t u^{(2)}]
=\left(\frac{r}{h}+T^{(1)}+T^{(2)}+P\right)\cdot Q_{r}(t)-Q_{r}(t)\cdot  \left(\delta T^{(1)}+\delta T^{(2)}+\delta P\right).
\end{align}
Let us write $Q_r=\sum_{i,j,k,l} Q_r^{ijkl}\otimes E_{ij}\otimes E_{kl}$ in terms of the basis $\{E_{ij}\otimes E_{kl}\}_{1\le i,j,k,l \le m+n}$ of ${\rm End}(\mathbb{C}^{m|n})^{\otimes 2}$, where each $Q_r^{ijkl}\in{\rm End}(V)$. Hereafter, we abbreviate $\rho(e_{ij})\in {\rm End}(V)$ as $e_{ij}$. %Recall that ${\rm End}(V)$ is finite dimensional, in contrast to $U(\gl_{m|n})$.

A careful sign-respecting computation then shows that the $ijkl$ component of the right hand side of \eqref{eq:Qm} equals
\begin{align}\label{eq:Qmijkl}
    \sum_{b}\left[(-1)^{\star_1}e_{ib}Q_r^{bjkl}+(-1)^{\star_2}e_{kb}Q_r^{ijbl}\right]-(-1)^{\star_3}Q_r^{kjil}-Q_r^{ijkl}\big[(-1)^{\bj}e_{jj}+(-1)^{\bl}e_{ll}-\delta_{jl}(-1)^{\bl}-\frac{r}{h}\big].
\end{align}
where the parity components are given by
\begin{equation*}
\star_1=\bi\bar{b}+(\bi+\bar{b})|Q_r^{bjkl}|;\quad
\star_2=\bk\bar{b}+(\bk+\bar{b})\left(|Q_r^{ijbl}|+\bi+\bj\right);\quad
\star_3=\bk\bj+\bk\bi+\bi\bj.  
\end{equation*}
In fact, the determination of $Q_r$ below will show that $Q_r^{ijkl}$ is homogeneous of parity $\bi+\bj+\bk+\bl$, analogously to Proposition \ref{prop:formal-solution-sec2}.

\smallskip
\noindent\textit{\textbf{Step 2. Determination when $t\in D_u$}}

Given $t\in D_u$ and the induction assumption that $Q_r^{ijkl}$ is defined for all indices $ijkl$, the $ijkl$ component of the left hand side of \eqref{eq:Qm}, $[(u_j-u_i)+t(u_l-tu_k)]Q_{r+1}^{ijkl}$, can uniquely determine $Q_{r+1}^{ijkl}$ as long as $i\neq j$ or $k\neq l$. We shall refer to these components $(ijkl)$ as the "off-diagonal" part. 

Next, for the diagonal case $i=j$ and $k=l$, we consider the equation $\eqref{eq:Qmijkl}=0$. Replacing $r,ijkl$ by $r+1,iill$ yields a new equation, in which the only undetermined terms are
\begin{align*}
    \left((-1)^{\bi} e_{ii}+(-1)^{\bl} e_{ll}\right)Q_{r+1}^{iill}-Q_{r+1}^{iill}\left((-1)^{\bi} e_{ii}+(-1)^{\bl} e_{ll}-(r+1)/h\right).
\end{align*}
This expression is exactly the operator
\begin{equation*}
    \mathcal{A}_r:=\adjoint_{(-1)^{\bi} e_{ii}+(-1)^{\bl} e_{ll}}+\frac{r+1}{h}{\rm Id}
\end{equation*}
acting on $Q_{r+1}^{iill}$. Since $h\notin \mathbb{Q}$, the eigenvalues of $\mathcal{A}_r$ are all nonzero, therefore $\mathcal{A}_r$ is invertible on the finite dimensional space ${\rm End}(V)$. This concludes that $Q_{r+1}^{iill}$ is uniquely determined by the resulting equation, which finishes the induction.

\smallskip
\noindent\textit{\textbf{Step 3. Solution satisfy \eqref{eq:super-irregular-kz-t} and extension to $t\neq 0,1$}}

Since the complement of $D_u$ is discrete, we first show that $\widehat{Y}(z,t)$ satisfies \eqref{eq:super-irregular-kz-t} on $t\in D_u$; then, using the fact that this equation has no singularities except at $t=0,1$, we naturally extend $\widehat{Y}(z,t)$ to all of $t\neq 0,1$. 

First note that $t^{h\delta T^{(2)}}(t-1)^{h\delta P}$ is independent of $z$; thus $\widehat{Y}(z,t)$ satisfies equation \eqref{eq:super-irregular-kz-z}. By the compatibility of \eqref{eq:super-irregular-kz-z} and \eqref{eq:super-irregular-kz-t}, the function
\begin{equation}\label{eq:Gzt}
G(z,t)=\frac{1}{h}\frac{\dd \widehat{Y}(z,t)}{\dd t}-\left(z u^{(2)}+\frac{T^{(2)}}{t}+\frac{P}{t- 1}\right)\cdot \widehat{Y}
\end{equation}
satisfies equation \eqref{eq:super-irregular-kz-z}. Since we have established the uniqueness of $\widehat{Y}(z,t)$ solving \eqref{eq:super-irregular-kz-z}, to prove that $G(z,t)$ vanishes, it suffices to check that $G(z,t)$ has the same form as that in \eqref{eq:FormalSolY}, with its leading term of the formal power series in $z^{-1}$ equal to zero.

Inserting \eqref{eq:FormalSolY} into \eqref{eq:Gzt} and using the commutativity of all diagonal supertensors, we extract from $G(z,t)$ the exponential factor
\begin{equation*}
e^{hz(u^{(1)}+tu^{(2)})}z^{h\delta (T^{(1)}+T^{(2)}+P)}t^{h\delta T^{(2)}}(t-1)^{h\delta P}.   
\end{equation*}
The remaining part is a formal power series in $z^{-1}$; its leading term is given by
\begin{equation*}
G_0(t)=\left[Q_1(t),u^{(2)}\right]+\frac{\delta T^{(2)}-T^{(2)}}{t}+\frac{\delta P-P}{t-1}.
\end{equation*}
On the other hand, taking $m=0$ in \eqref{eq:Qm} and $Q_0(t)\equiv\mathrm{Id}$ gives
\begin{equation*}
    [Q_1(t),u^{(1)}+t u^{(2)}]=T^{(1)}+T^{(2)}+P-\delta T^{(1)}-\delta T^{(2)}-\delta P.
\end{equation*}
Separating the terms in $u^{(1)}$ and $u^{(2)}$ yields
\begin{align*}
    \left[Q_1(t),u^{(1)}\right]&=\left(T^{(1)}-\delta T^{(1)}\right)+\frac{P-\delta P}{1-t}, \\
    \left[Q_1(t),u^{(2)}\right]&=\frac{T^{(2)}-\delta T^{(2)}}{t}+\frac{P-\delta P}{t-1}.
\end{align*}
This finishes the proof of the proposition.

%% file: final_bigThms/3.2-Lem-DecompY.tex
The formal solution $\widehat{Y}(z,t)$ in \eqref{eq:FormalSolY} have the following three decompositions corresponding to the $t-$poles $0,\infty, 1$. The explicit definition of the factor $(\frac{t-1}{1-t})^{h\delta P}$ is explained in Remark \ref{rmk:multivalue_1-t}.
\begin{enumerate}
\item[(a).] (Around $t=0$) 
$\widehat{Y}(z,t)=\ _1\widehat{W}(z,t)\widehat{F}^{(2)}(z,t)\cdot (\frac{t-1}{1-t})^{h\delta P}$, where $_1\widehat{W}$ and $\widehat{F}^{(2)}$ satisfy compatible systems
\begin{align}\label{eq:zt-funcW1F2}
    &\left\{
    \begin{aligned}
        \frac{1}{h} \frac{\partial _1\widehat{W}}{\partial z} &= \left(u^{(1)}+\frac{T^{(1)}+P}{z}\right)\ _1\widehat{W} + \left[tu^{(2)}+\frac{T^{(2)}}{z},\ _1\widehat{W}\right], \\
        \frac{1}{h} \frac{\partial _1\widehat{W}}{\partial t} &= \left(\frac{P}{t-1}\right)\ _1\widehat{W} + \left[zu^{(2)}+\frac{T^{(2)}}{t},\ _1\widehat{W}\right],
    \end{aligned}
    \right.
    &&
    \left\{
    \begin{aligned}
        \frac{1}{h}\frac{\partial \widehat{F}^{(2)}}{\partial z}&=\left(tu^{(2)}+\frac{T^{(2)}}{z}\right)\widehat{F}^{(2)}, \\
        \frac{1}{h}\frac{\partial \widehat{F}^{(2)}}{\partial t}&=\left(zu^{(2)}+\frac{T^{(2)}}{t}\right)\widehat{F}^{(2)},
    \end{aligned}
    \right.
\end{align}
and are uniquely determined by the forms
\begin{align}
    \label{eq:W1}
    _1\widehat{W}&=\ _1\widehat{K}\cdot e^{hzu^{(1)}}z^{h\delta(T^{(1)}+P)}, \\
    \label{eq:F2}
    \widehat{F}^{(2)}&=\widehat{H}^{(2)}\cdot e^{hztu^{(2)}}z^{h\delta T^{(2)}}t^{h\delta T^{(2)}}.
\end{align}
Here $_1\widehat{K}$ and $\widehat{H}^{(2)}$ are formal power series in $z^{-1}$ with leading terms $(1-t)^{h\delta P}=\mathrm{Id}+O(t)$ (see Remark \ref{rmk:multivalue_1-t}) and $\mathrm{Id}$ respectively. More explicitly, we have 
\[\widehat{Q}=\ _1\widehat{K}\widehat{H}^{(2)}\cdot (1-t)^{-h\delta P}.\]

\item[(b).](Around $t=\infty$) 
$\widehat{Y}(z,t)=\ _2\widehat{W}(z,t)\widehat{F}^{(1)}(z,t)$, where $_2\widehat{W}$ and $\widehat{F}^{(1)}$ satisfy compatible systems
\begin{align}\label{eq:zt-funcW2F1}
    &\left\{
    \begin{aligned}
        \frac{1}{h}\frac{\partial\ _2\widehat{W}}{\partial z}&=\left(tu^{(2)}+\frac{T^{(2)}+P}{z}\right)\ _2\widehat{W}+\left[u^{(1)}+\frac{T^{(1)}}{z},\ _2\widehat{W}\right],\\
        \frac{1}{h}\frac{\partial\ _2\widehat{W}}{\partial t}&=\left(zu^{(2)}+\frac{T^{(2)}}{t}+\frac{P}{t-1}\right)\ _2\widehat{W},
    \end{aligned}
    \right.
    &&
    \left\{
    \begin{aligned}
        \frac{1}{h}\frac{\partial \widehat{F}^{(1)}}{\partial z}&=\left(u^{(1)}+\frac{T^{(1)}}{z}\right)\widehat{F}^{(1)}, \\
        \frac{1}{h}\frac{\partial \widehat{F}^{(1)}}{\partial t}&=0,
    \end{aligned}
    \right.
\end{align}
and are uniquely determined by the forms
\begin{align}
    \label{eq:W2}
    _2\widehat{W}&=\ _2\widehat{K}\cdot e^{hztu^{(2)}}(zt)^{h\delta(T^{(2)}+P)}, \\
    \label{eq:F1}
    \widehat{F}^{(1)}&=\widehat{H}^{(1)}\cdot e^{hzu^{(1)}}z^{h\delta T^{(1)}}.
\end{align}
Here $_2\widehat{K}$ and $\widehat{H}^{(1)}$ are formal power series in $z^{-1}$ with leading terms $(1-\frac{1}{t})^{h\delta P}$ and $\mathrm{Id}$ respectively. More explicitly, we have 
\[\widehat{Q}=\ _2\widehat{K}\widehat{H}^{(1)}\cdot (1-\frac{1}{t})^{-h\delta P}.\]

\item[(c).](Around $t=1$)
$\widehat{Y}(z,t)=\widehat{U}(z,t)\widehat{X}(z,t)$, where $\widehat{U}$ and $\widehat{X}$ satisfy compatible systems
\begin{align}\label{eq:zt-funcUX}
    &\left\{
    \begin{aligned}
        \frac{1}{h}\frac{\partial \widehat{U}}{\partial z}&=\!\left(u^{(1)}\!+\!u^{(2)}\!+\!\frac{T^{(1)}\!+\!T^{(2)}}{z}\right)\widehat{U}\!+\!\left[(t-1)u^{(2)}\!+\!\frac{P}{z},\widehat{U}\right],\\
        \frac{1}{h}\frac{\partial \widehat{U}}{\partial t}&=\!\left(\frac{T^{(2)}}{t}\right)\widehat{U}+\left[zu^{(2)}+\frac{P}{t-1},\widehat{U}\right],
    \end{aligned}
    \right.
    &&
    \left\{
    \begin{aligned}
        \frac{1}{h}\frac{\partial \widehat{X}}{\partial z}&=\!\left((t-1)u^{(2)}\!+\!\frac{P}{z}\right)\widehat{X}, \\
        \frac{1}{h}\frac{\partial \widehat{X}}{\partial t}&=\!\left(zu^{(2)}+\frac{P}{t-1}\right)\widehat{X},
    \end{aligned}
    \right.
\end{align}
and  are uniquely determined by the forms
\begin{align}
    \label{eq:U}
    \widehat{U}&=\widehat{I}\cdot e^{hz(u^{(1)}+u^{(2)})}z^{h\delta(T^{(1)}+T^{(2)})}, \\
    \label{eq:X}
    \widehat{X}&=\widehat{L}\cdot e^{hz(t-1)u^{(2)}}(z(t-1))^{h\delta P}.
\end{align}
Here $\widehat{I}$ and $\widehat{L}$ are formal power series in $z^{-1}$ with leading terms $t^{h\delta T^{(2)}}$ and $\mathrm{Id}$ respectively. More explicitly, we have 
\[\widehat{Q}=\widehat{I}\widehat{L}\cdot t^{-h\delta T^{(2)}}.\]
\end{enumerate}

%% file: final_bigThms/3.2-Proof-Lem-DecompY.tex
(1) We first construct $_1\widehat{W}$ and $\widehat{F}^{(2)}$ in the prescribed form.
Substituting \eqref{eq:W1} and \eqref{eq:F2} into the $z$-equations of \eqref{eq:zt-funcW1F2}, we obtain differential equations for $_1\widehat{K}$ and $\widehat{H}^{(2)}$:
\begin{align}\label{eq:K1-z}
    & \frac{1}{h}\frac{\dd \ _1\widehat{K}}{\dd z}=\left(u^{(1)}+\frac{T^{(1)}+P}{z}\right)\ _1\widehat{K}-\ _1\widehat{K}\left(u^{(1)}+\frac{\delta T^{(1)}+\delta P}{z}\right)+\left[tu^{(2)}+\frac{T^{(2)}}{z},\ _1\widehat{K}\right],\\
    \label{eq:H2-z}
    & \frac{1}{h}\frac{\dd \widehat{H}^{(2)}}{\dd z}=\left(tu^{(2)}+\frac{T^{(2)}}{z}\right)\widehat{H}^{(2)}-\widehat{H}^{(2)}\left(tu^{(2)}+\frac{\delta T^{(2)}}{z}\right).
\end{align}
We look for solutions in the form of formal power series
\[
_1\widehat{K} = \sum_{m\ge 0} K_m z^{-m}, 
\qquad
\widehat{H}^{(2)} = \sum_{m\ge 0} H_m z^{-m},
\]
with leading terms
$
K_0 = (1-t)^{h\delta P}, 
\ 
H_0 = \mathrm{Id}.$

Inserting these expansions into \eqref{eq:K1-z}, \eqref{eq:H2-z} and comparing coefficients of $z^{-m-1}$, we obtain recursive relations which uniquely determine all coefficients $K_m$ and $H_m$. Following an argument similar to that in Proposition \ref{thm:FormalSolY}, we verify $_1\widehat{K}$ and $\widehat{H}^{(2)}$ exist and are uniquely determined. Consequently, $_1\widehat{W}$ and $\widehat{F}^{(2)}$ are well-defined and satisfy both the $z$- and $t$-equations in \eqref{eq:zt-funcW1F2}.

\medskip

We now compare $_1\widehat{W}\widehat{F}^{(2)}$ with $\widehat{Y}$. First, it is checked by Leibniz rule that once $_1\widehat{W}$ and $\widehat{F}^{(2)}$ satisfy \eqref{eq:zt-funcW1F2}, the product $_1\widehat{W}\widehat{F}^{(2)}$ satisfies the same system \eqref{eq:super-irregular-kz-z}–\eqref{eq:super-irregular-kz-t} as $\widehat{Y}$.
\begin{comment}
\rednote{need modification here} We now compare $_1\widehat{W}\widehat{F}^{(2)}$ with $\widehat{Y}$.
First, set
\[
M_1 = u^{(1)}+\frac{T^{(1)}+P}{z},\quad M_2 = t u^{(2)}+\frac{T^{(2)}}{z};
\qquad
N_1 = \frac{P}{t-1},\quad N_2 = z u^{(2)}+\frac{T^{(2)}}{t},
\]
then the product $_1\widehat{W}\widehat{F}^{(2)}$ satisfies the same system \eqref{eq:super-irregular-kz-z}–\eqref{eq:super-irregular-kz-t} as $\widehat{Y}$.
\end{comment}
Second, we examine their asymptotic forms. Using the commutativity of $u^{(1)}$, $\delta(T^{(1)}+P)$ with $u^{(2)}$, $T^{(2)}$, one rewrites
\begin{align*}
_1\widehat{W}\widehat{F}^{(2)}\Bigl(\frac{t-1}{1-t}\Bigr)^{h\delta P}
&=
\left(_1\widehat{K}\widehat{H}^{(2)}(1-t)^{-h\delta P}\right)
\cdot
e^{hz(u^{(1)}+t u^{(2)})}
z^{h\delta(T^{(1)}+T^{(2)}+P)}
t^{h\delta T^{(2)}}
(t-1)^{h\delta P}\\
&=\left(1+O(z^{-1})\right)
\cdot
e^{hz(u^{(1)}+t u^{(2)})}
z^{h\delta(T^{(1)}+T^{(2)}+P)}
t^{h\delta T^{(2)}}
(t-1)^{h\delta P}
\end{align*}
Thus $_1\widehat{W}\widehat{F}^{(2)}\big(\frac{t-1}{1-t}\big)^{h\delta P}$ has the same asymptotic form as $\widehat{Y}$, namely the same formal $z^{-1}$-expansion together with the same exponential factor. By the uniqueness statement of Proposition \ref{thm:FormalSolY}, we conclude that
\[
\widehat{Y}=\ _1\widehat{W}\widehat{F}^{(2)}\cdot\Bigl(\frac{t-1}{1-t}\Bigr)^{h\delta P}.
\]

(2) The second case is completely parallel. We record the key equations for $_2\widehat{K}$ and $\widehat{H}^{(1)}$:
\begin{align}\label{eq:K2-z}
&\frac{1}{h}\frac{\dd \ _2\widehat{K}}{\dd z}
=
\Bigl(tu^{(2)}+\frac{T^{(2)}+P}{z}\Bigr)\ _2\widehat{K}-\ _2\widehat{K}\Bigl(tu^{(2)}+\frac{\delta T^{(2)}+\delta P}{z}\Bigr)
+
\Bigl[u^{(1)}+\frac{T^{(1)}}{z},\ _2\widehat{K}\Bigr],
\\
\label{eq:H1-z}
&\frac{1}{h}\frac{\dd \widehat{H}^{(1)}}{\dd z}
=
\Bigl(u^{(1)}+\frac{T^{(1)}}{z}\Bigr)\widehat{H}^{(1)}-\widehat{H}^{(1)}\Bigl(u^{(1)}+\frac{\delta T^{(1)}}{z}\Bigr).
\end{align}
\begin{comment}
Together with
\[
M_1 = tu^{(2)}+\frac{T^{(2)}+P}{z},\quad
M_2 = u^{(1)}+\frac{T^{(1)}}{z},\quad
N_1 = z u^{(2)}+\frac{T^{(2)}}{t}+\frac{P}{t-1},\quad
N_2 = 0,
\]
\end{comment}
the conclusion follows by the same argument.

\medskip
\noindent
(3) The third case is also parallel. The relevant equations are
\begin{align}\label{eq:I-z}
\frac{1}{h}\frac{\dd \,\widehat{I}}{\dd z}
&=
\Bigl(u^{(1)}+u^{(2)}+\frac{T^{(1)}+T^{(2)}}{z}\Bigr)\widehat{I}
-
\widehat{I}\Bigl(u^{(1)}+u^{(2)}+\frac{\delta T^{(1)}+\delta T^{(2)}}{z}\Bigr)
+
\Bigl[(t-1)u^{(2)}+\frac{P}{z},\,\widehat{I}\Bigr],
\\
\label{eq:L-z}
\frac{1}{h}\frac{\dd \,\widehat{L}}{\dd z}
&=
\Bigl((t-1)u^{(2)}+\frac{P}{z}\Bigr)\widehat{L}
-
\widehat{L}\Bigl((t-1)u^{(2)}+\frac{\delta P}{z}\Bigr).
\end{align}
\begin{comment}
Together with
\[
M_1 = u^{(1)}+u^{(2)}+\frac{T^{(1)}+T^{(2)}}{z},\quad
M_2 = (t-1)u^{(2)}+\frac{P}{z},\quad
N_1 = \frac{T^{(2)}}{t},\quad
N_2 = z u^{(2)}+\frac{P}{t-1},
\]
\end{comment}
the result follows.

%% file: final_sections/section3/Section3.3.tex
\subsection{Holomorphicity of the first factor in the factorization}\label{sec:3.3}

In Section \ref{sec:3.2} we defined the formal solution $\widehat{Y}(z,t)$ of \eqref{eq:super-irregular-kz-z}-\eqref{eq:super-irregular-kz-t} and introduced three factorizations of $\widehat{Y}$. In this section, we prove that the factorizations of $\widehat{Y}$ completely extract the singularity of $\widehat{Y}$ at each pole into the second component. This extraction is realized by assigning each decomposition a local coordinate.

\begin{thm}\label{thm:Y-decomp}
The formal solution $\widehat{Y}(z,t)$ in \eqref{eq:FormalSolY} have the following decompositions, where in each cases we introduce the local coordinate $\sigma_1,\sigma_2$ and $\eta=\sigma_1\sigma_2$. We note that $\sigma_2$ approaches $0$ as $t$ approaches the corresponding poles $0,\infty, 1$ respectively.
\begin{enumerate}
\item[(a).](Around $t=0$) Set
$\sigma_1=z,\sigma_2=t$ and $\eta=zt$, under change of variables holds
\begin{equation*}\label{eq:coord-Y=W1F2}
\widehat{Y}(z,t)\cdot \left(\frac{t-1}{1-t}\right)^{-h\delta P}=
\ _1\widehat{W}(\sigma_1,\sigma_2)\widehat{F}^{(2)}(\eta);
\end{equation*}
where $\ _1\widehat{W}(\sigma_1,\sigma_2)=\ _1\widehat{K}(\sigma_1,\sigma_2)\cdot e^{h\sigma_1u^{(1)}}\sigma_1^{h\delta (T^{(1)}+P)}$ and $\widehat{F}^{(2)}(\eta)=\widehat{H}^{(2)}(\eta)\cdot e^{h\eta u^{(2)}}\eta^{h\delta T^{(2)}}$. Here $\ _1\widehat{K}$ and $\widehat{H}^{(2)}$ are formal power series in $\sigma_1^{-1}$ and $\eta^{-1}$ with leading terms $(1-\sigma_2)^{h\delta P}$ and $\mathrm{Id}$ respectively. 

\smallskip
Moreover, $\widehat{F}^{(2)}(\eta)$ satisfy
\begin{equation}\label{eq:monoF2}
    \frac{\dd \widehat{F}^{(2)}}{\dd \eta}=\left(u^{(2)}+\frac{T^{(2)}}{\eta}\right)\widehat{F}^{(2)}.
\end{equation}

\smallskip
Meanwhile, the coefficients of $\ _1\widehat{K}(\sigma_1,\sigma_2)$ as a function in $\sigma_2$ are holomorphic at $\sigma_2=0$.

\item[(b).](Around $t=\infty$)
Set $\sigma_1=zt,\sigma_2=1/t$ and $\eta=z$, under change of variables holds
\begin{equation*}\label{eq:coord-Y=W2F1}
\widehat{Y}(z,t)=
\ _2\widehat{W}(\sigma_1,\sigma_2)\widehat{F}^{(1)}(\eta);
\end{equation*}
where $\ _2\widehat{W}(\sigma_1,\sigma_2)=\ _2\widehat{K}(\sigma_1,\sigma_2)\cdot e^{h\sigma_1u^{(2)}}\sigma_1^{h\delta (T^{(2)}+P)}$ and $\widehat{F}^{(1)}(\eta)=\widehat{H}^{(1)}(\eta)\cdot e^{h\eta u^{(1)}}\eta^{h\delta T^{(1)}}$. Here $\ _2\widehat{K}$ and $\widehat{H}^{(1)}$ are formal power series in $\sigma_1^{-1}$ and $\eta^{-1}$ with leading terms $(1-\sigma_2)^{h\delta P}$ and $\mathrm{Id}$ respectively.

\smallskip
Moreover, $\widehat{F}^{(1)}(\eta)$ satisfy
\begin{equation}\label{eq:monoF1}
    \frac{\dd \widehat{F}^{(1)}}{\dd \eta}=\left(u^{(1)}+\frac{T^{(1)}}{\eta}\right)\widehat{F}^{(1)}.
\end{equation}

\smallskip
Meanwhile, the coefficients of $\ _2\widehat{K}(\sigma_1,\sigma_2)$ as a function in $\sigma_2$ are holomorphic at $\sigma_2=0$.

\item[(c).](Around $t=1$)
Set $\sigma_1=z,\sigma_2=t-1$ and $\eta=zt-z$, under change of variables holds
\begin{equation*}\label{eq:coord-Y=UX}
\widehat{Y}(z,t)=
\widehat{U}(\sigma_1,\sigma_2)\widehat{X}(\eta);
\end{equation*}
where $\widehat{U}(\sigma_1,\sigma_2)=\widehat{I}(\sigma_1,\sigma_2)\cdot e^{h\sigma_1(u^{(1)}+u^{(2)})}\sigma_1^{h\delta (T^{(1)}+T^{(2)})}$ and $\widehat{X}(\eta)=\widehat{L}(\eta)\cdot e^{h\eta u^{(2)}}\eta^{h\delta P}$. Here $\widehat{U}$ and $\widehat{X}$ are formal power series in $\sigma_1^{-1}$ and $\eta^{-1}$ with leading terms $(1+\sigma_2)^{h\delta T^{(2)}}$ and $\mathrm{Id}$ respectively.

\smallskip
Moreover, $\widehat{X}(\eta)$ satisfy
\begin{equation}\label{eq:monoX}
    \frac{\dd \widehat{X}}{\dd \eta}=\left(u^{(2)}+\frac{P}{\eta}\right)\widehat{X}.
\end{equation}

\smallskip
Meanwhile, the coefficients of $\widehat{I}(\sigma_1,\sigma_2)$ as a function in $\sigma_2$ are holomorphic at $\sigma_2=0$.

\end{enumerate}
\end{thm}

\begin{proof}
Only the statements of holomorphicity are need to be proved, the rest reads directly from Proposition \ref{lem:DecompY}. We give the full proof for $_1\widehat{K}$, then as we will show, the remaining cases (b), (c) are very much the same.

For case $_1\widehat{K}$, $(\sigma_1,\sigma_2)=(z,t)$. Hence the $\sigma_1-$equation for $_1\widehat{K}$ in this case is exactly \eqref{eq:K1-z}, and the $\sigma_2-$equation for $_1\widehat{K}$ is obtained from inserting the expression \eqref{eq:W1} into \eqref{eq:zt-funcW1F2}:
\begin{equation}\label{eq:K1-t}
    \frac{1}{h} \frac{\partial\ _1\widehat{K}}{\partial t} = \left(\frac{P}{t-1}\right)\ _1\widehat{K} + \left[zu^{(2)}+\frac{T^{(2)}}{t},\ _1\widehat{K}\right].
\end{equation}
Now substituting $_1\widehat{K}(z,t)=\sum_{r\ge 0}K_r(t)\,z^{-r}$ gives rise to the $r$-recursion for $z$:
\begin{equation}\label{eq:K-rec-z}
\big[K_{r+1}(t),\;u^{(1)}+t u^{(2)}\big]
= \Big(\frac{r}{h}+T^{(1)}+T^{(2)}+P\Big)K_r(t)
  -K_r(t)\big(\delta T^{(1)}+T^{(2)}+\delta P\big),
\end{equation}
and for $t$:
\begin{equation}\label{eq:K-rec-t}
\frac{1}{h}\frac{\partial K_r}{\partial t}
= \frac{P}{t-1}\cdot K_r
+\Big[\frac{T^{(2)}}{t},\,K_r\Big]
+\big[u^{(2)},\,K_{r+1}\big].
\end{equation}
We write $K_r=\sum_{i,j,k,l}K_r^{ijkl}\otimes E_{ij}\otimes E_{kl}$
throughout. The proof proceeds in two parts: we first analyze the
recursion at $t=0$, and then prove holomorphicity by induction on $r$.

\medskip
\noindent\textbf{The recursion at $t=0$.}
Setting $t=0$ in \eqref{eq:K-rec-z} with $K_0'=\mathrm{Id}$ gives
\begin{equation}\label{eq:K-rec-0}
\big[K_{r+1}',\;u^{(1)}\big]
= \Big(\frac{r}{h}+T^{(1)}+T^{(2)}+P\Big)K_r'
  -K_r'\big(\delta T^{(1)}+T^{(2)}+\delta P\big).
\end{equation}

\begin{lem}\label{lem:K-at-zero}
The recursion \eqref{eq:K-rec-0} determines a unique sequence $\{K_r'\}_{r\ge 0}$, and each $K_r'$ commutes with $T^{(2)}$.
\end{lem}

\begin{proof}
The proof is similar to Proposition \ref{prop:formal-solution-sec2} and \ref{thm:FormalSolY}. Suppose $K_0',\cdots K_r'$ are determined, we consider $K_{r+1}'$. For $i\neq j$, the entries $K_{r+1,ijkl}$ are determined by reading the $(ijkl)-$entry on both sides of~\eqref{eq:K-rec-0}, which gives 
\begin{equation}\label{eq:condition-ijkl}
(u_j-u_i)K'_{r+1,ijkl}=\left[\Big(\frac{r}{h}+T^{(1)}+T^{(2)}+P\Big)K_r'
  -K_r'\big(\delta T^{(1)}+T^{(2)}+\delta P\big)\right]_{ijkl},\quad \forall i\neq j,\,k,l.
\end{equation}
As for $i=j$, the entries $K_{r+1,iikl}$ is determined by replacing $r$ by $r+1$ on the right hand side of \eqref{eq:K-rec-0}, taking the $(iikl)-$ entries on both sides and solving the identities
\begin{equation}\label{eq:condition-iikl}
    \left[\Big(\frac{r}{h}+T^{(1)}+T^{(2)}+P\Big)K_{r+1}'
  -K_{r+1}'\big(\delta T^{(1)}+T^{(2)}+\delta P\big)\right]_{iikl}=0,\quad \forall i,k,l.
\end{equation}
Actually, \eqref{eq:condition-iikl} is equivalent to the condition of  $X_{r+1}^i:=\sum_{k,l}K'_{r+1,iikl}\otimes E_{kl}$ satisfying the equation
\begin{equation}\label{eq:determine-iikl}
\left[\frac{r+1}{h}\,\mathrm{Id}
+\mathrm{ad}_{e_{ii}\otimes 1}
-\mathrm{ad}_{1\otimes E_{ii}}+\mathrm{ad}_{\sum_{a,b}e_{ab}\otimes E_{ab}}\right]X_{r+1}^i
=
\mathsf{R}_{r+1}^i,
\end{equation}
where $\mathsf{R}_{r+1}^i$ is an expression containing only $K_{r+1,ijkl}$ with $i\neq j$, hence is already determined. Since $h\notin\mathbb{Q}$ and $V$ is finite-dimensional, the operator on the left hand side of \eqref{eq:determine-iikl} has no eigenvalue in $\mathbb{Z}$, meaning it is invertible. Thus $X_{r+1}^i$ is uniquely determined.

For the commutativity of $K_r'$ and $T^{(2)}$, insert $t=0$ into the $z-$equation of \eqref{eq:zt-funcW1F2}, one derives the equation for $W'=\sum_{r\geq 0}K_r'\cdot z^{-r}$ to be 
\begin{equation*}
\frac{1}{h} \frac{\partial W'}{\partial z} = \left(u^{(1)}+\frac{T^{(1)}+P}{z}\right)W' + \left[\frac{T^{(2)}}{z},W'\right].
\end{equation*}
Since $u^{(1)}$ and $T^{(1)}+P$ all commute with $T^{(2)}$, formal gauge transformation of $\exp(hT^{(2)}\ln z)$ can be applied. Furthermore, $\delta(T^{(1)}+P)$ also commutes with $T^{(2)}$; this leads to $W'$ and $K'$ all commute with $T^{(2)}$.
\end{proof}

\medskip
\noindent\textbf{Proof of Theorem \ref{thm:Y-decomp}.}

We prove by induction on $r$ the following two propositions.
\begin{equation*}
    (\mathscr{H})_r: \, K_r(t) \text{ is holomorphic at } t=0;\qquad (\mathscr{L})_r: \, \lim_{t\to 0}K_r(t)=K_r'.
\end{equation*}
where $\{K_r'\}_{r\geq 0}$ is the sequence defined from Lemma \ref{lem:K-at-zero}. The base case $r=0$ is clear since $K_0(t)=(1-t)^{h\delta P}$ is holomorphic at $t=0$ with $\lim_{t\to 0}K_0(t)=\mathrm{Id}=K_0'$. Assume $(\mathscr{H})_r$ and $(\mathscr{L})_r$ hold for $1,\cdots r$.

\smallskip
\noindent\textit{Holomorphy $(\mathscr{H})_{r+1}$.}
Expand near $t=0$ that $K_{r+1}(t)=C_{-1}/t+C_0+O(t)$, insert into \eqref{eq:K-rec-z} and take $t=0$, by $(\mathscr{H})_r$ and $(\mathscr{L})_r$ one derives
\begin{subequations}\label{eq:Holomorphic-expansion}
\begin{align}\label{eq:expansion1}
    [C_{-1},u^{(1)}]&=0;\\
    \label{eq:expansion2}
    [C_{-1},u^{(2)}]+[C_0, u^{(1)}]&=\Big(\frac{r}{h}+T^{(1)}+T^{(2)}+P\Big)K_r'
  -K_r'\big(\delta T^{(1)}+T^{(2)}+\delta P\big)
  =[K_{r+1},u^{(1)}].
\end{align}
\end{subequations}
Now \eqref{eq:expansion1} gives $C_{-1,ijkl}=0$ for $i\neq j$, \eqref{eq:expansion2} gives $C_{-1,iikl}=0$ for $k\neq l$, while similarly $C_{-1,iikk}$ is determined by an invertible linear operator applied to the entries $C_{-1,ijkl}$ and $C_{-1,iikl}$. Hence, $C_{-1}=0$ and $(\mathscr{H})_{r+1}$ is proved.

\smallskip
\noindent\textit{Limit $(\mathscr{L})_{r+1}$.}
We want to prove $C_0=K_{r+1}'$. Insert $C_{-1}=0$ into \eqref{eq:expansion2} we have 
\begin{equation}\label{eq:Limit-expansion1}
    [C_0, u^{(1)}]=[K_{r+1},u^{(1)}]=\Big(\frac{r}{h}+T^{(1)}+T^{(2)}+P\Big)K_r'
  -K_r'\big(\delta T^{(1)}+T^{(2)}+\delta P\big).
\end{equation}
On the other hand, multiply both sides of \eqref{eq:K-rec-t} by $t$ and insert it into \eqref{eq:K-rec-z} at level $r+1$, by eliminating $[K_{r+2}(t), u^{(2)}]$ term one derives
\begin{equation}\label{eq:Limit-expansion2}
    [K_{r+2}(t),u^{(1)}]=\Big(\frac{r}{h}+T^{(1)}+P\Big)K_{r+1}(t)
  -K_{r+1}(t)\big(\delta T^{(1)}+\delta P\big)+t\Big(\frac{1}{h}\frac{\partial K_{r+1}}{\partial t}-\frac{tP}{t-1}\cdot K_{r+1}\Big).
\end{equation}
Having established $(\mathscr{H})_{r+1}$, letting $t=0$ and taking the $(iikl)-$entries on both sides of \eqref{eq:Limit-expansion2} leads to
\begin{equation}\label{eq:Limit-expansion3}
    \left[\Big(\frac{r}{h}+T^{(1)}+P\Big)C_0
  -C_0\big(\delta T^{(1)}+\delta P\big)\right]_{iikl}=0,\quad \forall i,k,l.
\end{equation}
Eventually, by the commutativity of $K_r'$ and $T^{(2)}$ proved in Lemma \ref{lem:K-at-zero}, the equations \eqref{eq:Limit-expansion1}-\eqref{eq:Limit-expansion3} satisfied by $C_0$ are exactly the defining relations \eqref{eq:condition-ijkl}-\eqref{eq:condition-iikl} of $K_{r+1}'$. This proves $(\mathscr{L})_{r+1}$ and finishes the induction for case (1). 

For the case $_2\widehat{K}$, substituting \eqref{eq:W2} into \eqref{eq:zt-funcW2F1}, the equations for $\sigma_1=zt, \,\sigma_2=z$ becomes
\begin{align*}
    & \frac{1}{h}\frac{\partial\ _2\widehat{K}}{\partial\sigma_1}=\left(u^{(2)}+\frac{T^{(2)}+P}{\sigma_1}\right)\ _2\widehat{K}-\ _2\widehat{K}\left(u^{(2)}+\frac{\delta T^{(2)}+\delta P}{\sigma_1}\right)+\left[\sigma_2u^{(1)}+\frac{T^{(1)}}{\sigma_1},\ _2\widehat{K}\right],\\
    & \frac{1}{h} \frac{\partial\ _2\widehat{K}}{\partial \sigma_2} = \left(\frac{P}{\sigma_2-1}\right)\ _2\widehat{K} + \left[\sigma_1u^{(1)}+\frac{T^{(1)}}{\sigma_2},\ _2\widehat{K}\right],
\end{align*}
which is the dual of \eqref{eq:K1-z}-\eqref{eq:K1-t}. Thus the dual proof applies. For the case $\widehat{I}$, the equations for $\sigma_1\!=\!z,\,\sigma_2\!=\!t\!-\!1$ are
\begin{align*}
    & \frac{1}{h}\frac{\partial\ \widehat{I}}{\partial\sigma_1}\!=\!\left(u^{(1)}\!+\!u^{(2)}\!+\!\frac{T^{(1)}\!+\!T^{(2)}}{\sigma_1}\right)\widehat{I}\!-\!\widehat{I}\left(u^{(1)}\!+\!u^{(2)}\!+\!\frac{\delta (T^{(1)}\!+\!T^{(2)})}{\sigma_1}\right)\!+\!\left[\sigma_2u^{(2)}\!+\!\frac{P}{\sigma_1},\ \widehat{I}\right],\\
    & \frac{1}{h} \frac{\partial\ \widehat{I}}{\partial \sigma_2} = \left(\frac{T^{(2)}}{\sigma_2+1}\right)\widehat{I} + \left[\sigma_1u^{(2)}+\frac{P}{\sigma_2},\widehat{I}\right],
\end{align*}
and the proof is parallel using the corresponding commutativity $\big[P,T^{(1)}+T^{(2)}\big]=\big[P,\delta (T^{(1)}+T^{(2)})\big]=0.$
\end{proof}
%We conclude by summarizing the coordinate regimes in Table \ref{pic:DecompCoord}, where $\sigma_1,\sigma_2$ are the variables of the first-factor function and $\eta$ is the variable on which the second factor depends.

%\input{final_pictures/tab_DecompCoord}

%% file: final_sections/section3/Section3.4.tex
\subsection{Isomonodromy arguments and resummation on different regions}
\label{sec:3.4}

In the rest of the paper, we always take $u\in\mathfrak{h}_{\mathrm{reg}}$ and $h\notin \mathbb{Q}$, denote $\mathbb{C}^*=\mathbb{C}\setminus\{0\}$ and $\mathbb{C}^{\star}=\mathbb{C}\setminus\{0,1\}$.
The formal decompositions of Proposition~\ref{lem:DecompY} are identities between formal power series in~\(z^{-1}\). In this subsection we lift them to identities between holomorphic solutions obtained by Borel-Laplace summation. To this end, we examine the anti-Stokes directions with respect to $z$ within the framework of $t-$isomonodromic deformations.

% ----------------------------------------------------------------
\subsubsection{Resummation of
\texorpdfstring{$\widehat Y(z,t)$}{Y-hat} with respect to $z$}
\label{sec:4.4.Y}

For any  \(t\in\mathbb{C}\setminus\{0,1\}\), the leading irregular
term of~(\ref{eq:super-irregular-kz-z}) is
\[u^{(1)}+tu^{(2)}
=\sum_{i,k}(u_i+tu_k)\,1\otimes E_{ii}\otimes E_{kk},\]
with eigenvalues \(\lambda_{(i,k)}(t)=u_i+tu_k\). As in Definition \ref{def:anti-stokes-sec2}, the anti-Stokes lines of ~(\ref{eq:super-irregular-kz-z}) is denoted by $ {\rm aSR}(hu^{(1)}+htu^{(2)})$.
Let 
\[\mathcal{R}_{\infty}
=\{-\mathrm{Arg}(h(u_i-u_j))
+2m\pi:i\neq j\}.\]
For given fixed $u\in\h_{\rm reg}$ and $d\notin \mathcal{R}_{\infty}$, we introduce two sets of curves on the $t$-plane:
\begin{align*}
\mathcal{R}_0(d)&=\left\{t:{\rm arg}(t)+d=-{\rm Arg}\left(h(u_i-u_j)\right)+{2k\pi}: k\in \mathbb{Z}, i\neq j\right\},\\
\mathcal{R}_1(d) &= \left\{t:{\rm arg}(1-t)+d=-{\rm Arg}\left(h(u_i-u_j)\right)+{2k\pi}: k\in \mathbb{Z}, i\neq j\right\}. 
\end{align*}
Then $\mathcal{R}_0(d)\cup \mathcal{R}_1(d)$ cut the universal covering space $\widetilde{\mathbb{C}^{\star}}$ of $\mathbb{C}^{\star}$ into connected components, called $t-$chambers. Particularly, each of the intervals $(-\infty,0)$, $(0,1)$ and $(1,+\infty)$ are included a chamber. However, to uniquely specify their lifts to the covering space, one must fix the arguments $\arg t$ and $\arg (t-1)$. See Definition \ref{def:Six-solutions}.

\begin{thm}\label{thm:Y-on-chamber}
Given $d\notin \mathcal{R}_{\infty}$, for every \(t\)-chamber~\(D\subset \widetilde{\mathbb{C}^{\star}}\), there exists a unique holomorphic solution \(Y_{d,D}(z,t)\) of~(\ref{eq:super-irregular-kz-z})-(\ref{eq:super-irregular-kz-t}) on \(\widetilde{\mathbb{C}^{\ast}}\times D\), with the prescribed asymptotic behavior
\begin{equation*}
    Y_{d,D}(z,t)\cdot e^{-hz(u^{(1)}+tu^{(2)})}z^{-h\delta T^{(1)}} (zt)^{-h\delta T^{(2)}}(zt-z)^{-h\delta P} \sim \widehat{Q}(z;t) 
\end{equation*}
as \(z\to\infty\) in \(\arg(z)\in(d-\frac\pi 2,d+\frac\pi 2)\), uniformly on compact subsets
of~\(D\).
\end{thm}
\begin{proof}
For the given $d\notin \mathcal{R}_{\infty}$, and a point \[t_0=\frac{u_i-u_j}{u_l-u_k}\ne 0, 1, \infty, \ \text{for some fixed indices $i,j,k,l$},\] let us consider the curve determined by (for a fixed integer $m$)
\[ d=-\mathrm{Arg}(hu_i-hu_j+t(hu_k-hu_l))+{2m\pi}.\]
Denote by $\Sigma$  the branch of the curve containing $t_0$ as a limit point. Then the branch $\Sigma$ is not contained in the set $\mathcal{R}(hu;d)_0\cup \mathcal{R}(hu;d)_1$ of curves. 

Since the goal is to show that the solutions defined from different regions of a chamber coincide, it suffices to assume that $u_1,...,u_n$ in a generic position. The general case then follows immediately by a small perturbation of $u$, thanks to the holomorphic dependence on parameters.  Let us assume that for any integer $m'$ and any ordered indices $(a,b,c,d)\ne(i,j,k,l)$
\[d\ne -\mathrm{Arg}(hu_a-hu_b+t_0(hu_c-hu_d))+{2m'\pi}.\]
Then, since $hu_i-hu_j+t_0(hu_k-hu_l)=0$ (the corresponding eigenvalues $hu_i+t_0hu_k$ and $hu_j+t_0hu_l$ of the irregular term $hu^{(1)}+t_0hu^{(2)}$ equal to each other), the direction  $d$ is not an anti-Stokes direction of the equation  at $t=t_0$. Thus, if $Y_{d-,t_0}(z;t_0)$ and $Y_{d+,t_0}(z;t_0)$ are the resummation of $\widehat{Y}$ along two directions $d_-$ and $d_+$ that are sufficiently close to $d$ and separated by $d$, then $Y_{d-,t_0}(z;t_0)=Y_{d+,t_0}(z;t_0)$. 

Now if $t_1\in\Sigma$ is sufficiently close to $t_0$, then there exists a path $\gamma\subset \Sigma\cup\{t_0\}$ from $t'$ to $t_0$ such that $d_-,d_+\notin {\rm aSR}(hu^{(1)}+htu^{(2)})$ for all $t\in \gamma$ (i.e., during the movement along the path, no anti-Stokes ray intersects $d-$ and $d+$). Thus both resummation $Y_{d-,t}(z;t)$ and $Y_{d+,t}(z;t)$ are defined for any $t\in \gamma$. 
It follows from Proposition \ref{thm:FormalSolY} that the coefficient matrices $Q_m(t)$ of the formal power series part $\widehat{Q}(z;t)$ of $\widehat{Y}$ only have poles at $t=0,1,\infty$ (that is smooth at $t_0$), thus $Y_{d-,t}(z;t)$ and $Y_{d+,t}(z;t)$ depend smoothly on $t\in \gamma$. 
Since they are solutions to the common linear system of PDEs , the equality of initial condition $Y_{d-,t_0}(z;t_0)=Y_{d+,t_0}(z;t_0)$ implies $Y_{d-,t'}(z;t')=Y_{d+,t'}(z;t')$. Therefore, the Stokes factor 
\[\mathcal{K}_d(t'):=Y_{d-,t'}(z;t')^{-1}Y_{d+,t'}(z;t'),\] associated to the anti-Stokes ray $d=-\mathrm{Arg}(hu_i-hu_j+t'(u_k-u_l))+{2m\pi}$ of the equation \eqref{eq:super-irregular-kz-z} at $t=t'$, is the identity for all $t'\in \Sigma$ sufficiently close to $t_0$. Equivalently, it states that $d$ is an admissible direction of $\widehat{Y}(z;t)$ for $t\in \Sigma$ and close to $t_0$.

Then Let us assume that $t_1,...,t_k\in \Sigma$ are the points such that there exist some other indices $(a_s,b_s,c_s,d_s)\ne (i,j,k,l)$ and integer $m_s$ such that 
\begin{align}\label{eq:d wide range}
 d=-\mathrm{Arg}(hu_{a_s}-hu_{b_s}+t_s(hu_{c_s}-hu_{d_s}))+{2m_s\pi}, \ s=1,...,k.
\end{align}
These points are arranged in the natural order along the curve such that $t_1$ is next to $t_0$. Then we can divide the curve \[\Sigma=\Sigma_1\cup \{t_1\} \cup \Sigma_2\cup \{t_2\}\cup\cdots \cup \{t_{k-1}\}\cup \Sigma_{k}\]
into open segments $\Sigma_s$. The Stokes factor $\mathcal{K}_d(t)$ associated to $d$ is constant for $t$ in each open piece $\Sigma_i$. In particular, we have seen that $\mathcal{K}_d(t)=1$ for $t\in \Sigma_1$. 

If we let $t$ vary over a large range on the curve $\Sigma$, for example from $\Sigma_1$ to $\Sigma_2$, some other anti-Stokes ray may cross the direction $d$. Without loss of generality, we  may assume that there exists only one $d'(t)\in {\rm aSR}(hu^{(1)}+htu^{(2)})$ such that $d'(t)\ne d$ for $t\in \Sigma_1$, and $d'(t_1)$ at $t_1$ coincides with the direction $d$.
Take a small neighborhood $B(t_1)$ of $t_1$ such that there is no other anti-Stokes rays crossing $d$ and $d'(t)$ when $t\in B(t_1)$, and let $t'\in \Sigma_1\cap B(t_1),\ t''\in\Sigma_2\cap B(t_1)$.  Then following the definition of Stokes factors, we must have (depending on which side of 
$d$ the direction $d'(t)$ approaches from)
\begin{equation}
\mathcal{K}_d(t')\cdot \mathcal{K}_{d'(t')}(t')=\mathcal{K}_{d'(t'')}(t'')\cdot \mathcal{K}_d(t'')
\end{equation}
or
\begin{equation}
\mathcal{K}_d(t'')\cdot \mathcal{K}_{d'(t'')}(t'')=\mathcal{K}_{d'(t')}(t')\cdot \mathcal{K}_d(t').
\end{equation}
According to the unipotent property of the Stokes factors (similar to Propsition \ref{prop:triangularity-sec2}) 
\[
 \mathcal{K}_{d'(t')}(t')= 1+ \alpha'\cdot\mathrm{E}_{b_1a_1}\otimes\mathrm{E}_{d_1c_1}, \quad \mathcal{K}_{d'(t'')}(t'')= 1+ \alpha''\cdot\mathrm{E}_{b_1a_1}\otimes\mathrm{E}_{d_1c_1}, \quad \mathcal{K}_{d}(t'')= 1+ \beta\cdot\mathrm{E}_{ji}\otimes\mathrm{E}_{lk}.
\]
Since $\mathcal{K}_d(t')=1$, and $(a_1,b_1,c_1,d_1)\ne (i,j,k,l)$, we have $\mathcal{K}_d(t'')$ at $t_2$ will remain the identity matrix. 
It concludes that $d$ is an admissible direction of $\widehat{Y}(z;t)$ for all $t\in \Sigma$ but $t\notin \mathcal{R}(hu;d)_0\cup \mathcal{R}(hu;d)_1$. 

Therefore, for any connected t-chamber $D$ cut out by the curves $\mathcal{R}(hu;d)_0\cup \mathcal{R}(hu;d)_1$, the resummation $Y_{d,t}(t)$ is well defined for all $t\in{D}$. Then the solution $Y_{d,{D}}(z,t)$, regarded as the continuation of $Y_{d}(z,t)$ from $t\in {D}$ to $\mathbb{C}\setminus\{0,1\}$, has the required property. \end{proof}

\begin{rmk}\label{rmk:Y-rays}
    Theorem \ref{thm:Y-on-chamber} is equivalent to the following description of anti-Stokes rays for fixed $t$. Let
    \begin{align*}
    \mathcal{R}^0(t)
    &=\{-\mathrm{Arg}(ht(u_i-u_j))+2k\pi:i\neq j\}
    &=\{d\mid t\in \mathcal{R}_i(d)\},\\
    \mathcal{R}^1(t)
    &=\{-\mathrm{Arg}(h(t-1)(u_i-u_j))+2k\pi:i\neq j\}
    &=\{d\mid t\in \mathcal{R}_i(d)\}.
    \end{align*}
    be the sets corresponding to $\mathcal{R}_0(d),\mathcal{R}_1(d)$. Then the anti-Stokes rays of \eqref{eq:super-irregular-kz-z} are exactly the set $\mathcal{R}_{\infty}\cup\mathcal{R}^0(t)\cup\mathcal{R}^1(t)$. All other $d$ is admissible.
    Moreover, the three families $\mathcal{R}_{\infty},\mathcal{R}^0(t),\mathcal{R}^1(t)$ are exactly the anti-Stokes rays in the variables $z,zt,z(t-1)$ governing the three decompositions of Proposition~\ref{lem:DecompY}; see Proposition~\ref{pro:W-on-chamber}.
\end{rmk}

% ----------------------------------------------------------------
\subsubsection{Resummation of \texorpdfstring{$_1\widehat{W},\ _2\widehat{W},\ \widehat{U}$}{W1-hat,W2-hat,U-hat} under respective coordinates}

The first factors \({}_1\widehat W\), \({}_2\widehat W\),
\(\widehat U\) satisfy systems of the form
\(\frac{1}{h}\partial_z F
=M_1(z,t)\cdot_s F+[M_2(z,t),F]_s\),
with data \((M_1,M_2)\) as in Proposition \ref{lem:DecompY}. As these are linear systems, the theory of irregular singularities and Stokes phenomenon applies to them as well. 
%In fact, each can be recast as a standard linear system on a larger space; we refer the reader to Appendix A for details.

Applying the holomorphicity of the second variable in Theorem \ref{thm:Y-decomp}, under respective coordinates we have

\begin{pro}\label{pro:W-on-chamber}
Given \(d\notin\mathcal{R}_{\infty}\), there exists a small neighbourhood of \(t=0\) (resp.\
\(t=\infty\), resp.\ \(t=1\)), namely $D_0$ (resp. $D_{\infty}$, resp. $D_1$), such that
\begin{enumerate}
\item[(a).](Set $\sigma_1\!=\!z,\sigma_2\!=\!t$)
For $t\in D_0$, the resummation of $_1\widehat{W}(\sigma_1,\sigma_2)$ with respect to $\sigma_1$ along direction $d$ is well defined on $\widetilde{\mathbb{C}^{\ast}}$ with the asymptotic behavior $_1W(\sigma_1,\sigma_2)\cdot e^{-h\sigma_1u^{(1)}}\sigma_1^{-h\delta (T^{(1)}+P)}\sim\  _1\widehat{K}(\sigma_1,\sigma_2)$;

\item[(b).](Set $\sigma_1\!=\!zt,\sigma_2\!=\!1/t$)
For $t\in D_{\infty}$, the resummation of $_2\widehat{W}(\sigma_1,\sigma_2)$ with respect to $\sigma_1$ along direction $d$ is well defined on $\widetilde{\mathbb{C}^{\ast}}$ with the asymptotic behavior $_2W(\sigma_1,\sigma_2)\cdot e^{-h\sigma_1u^{(2)}}\sigma_1^{-h\delta (T^{(2)}+P)}\sim\  _2\widehat{K}(\sigma_1,\sigma_2)$;

\item[(c).](Set $\sigma_1\!=\!z,\sigma_2\!=\!t-1$)
For $t\in D_1$, the resummation of $\widehat{U}(\sigma_1,\sigma_2)$ with respect to $\sigma_1$ along direction $d$ is well defined on $\widetilde{\mathbb{C}^{\ast}}$ with the asymptotic behavior $U(\sigma_1,\sigma_2)\cdot e^{-h\sigma_1(u^{(1)}+u^{(2)})}\sigma_1^{-h\delta (T^{(1)}+T^{(2)})}\sim \widehat{I}(\sigma_1,\sigma_2)$.
\end{enumerate}

\end{pro}

\begin{proof}
We prove~\(\mathrm{(i)}\) where $(\sigma_1,\sigma_2)=(z,t)$; the other cases are analogous under respective coordinates. Since $_1\widehat{K}(\sigma_1, \sigma_2)$ is holomorphic at $\sigma_2 = 0$, and $T^{(2)}$ commutes with the solution
$_1\widehat{K}(\sigma_1;\sigma_2 = 0)$ by Lemma \ref{lem:K-at-zero}, substituting $\sigma_2 = 0$  into the $\sigma_2$-equation for $_1W(\sigma_1,\sigma_2)$ in \eqref{eq:zt-funcW1F2}, we have
\begin{equation}\label{eq:1W-reduced}
  \frac{1}{h}\,\frac{\partial\, _1\!W}{\partial \sigma_1}\bigg|_{\sigma_2=0}
  \;=\;
  \Bigl( u^{(1)} + \frac{T^{(1)} + P}{\sigma_1} \Bigr)\; _1\!W(\sigma_1;\sigma_2 = 0).
\end{equation}
This is a standard meromorphic linear system 
with irregular leading term $u^{(1)}$
and regular singular part $(T^{(1)}+P)/\sigma_1$.
The anti-Stokes directions  are
precisely
\[
  \bigl\{-\mathrm{Arg}\bigl(h(u_i - u_j)\bigr) : i \neq j\bigr\}
  \;=\; R_\infty.
\]
In particular, every direction $d \notin R_\infty$ is admissible
for the system \eqref{eq:1W-reduced}.
Since $_1W(\sigma_1,\sigma_2)$ also satisfies a compatible system for $\sigma_2$, we can use the same argument as in the proof of
Proposition \ref{thm:Y-on-chamber} (c.f. Remark \ref{rmk:Y-rays}) to conclude that the direction $d$ remains admissible
for the resummation of $_1\widehat{W}(\sigma_1, \sigma_2)$
with respect to $\sigma_1$ for all $\sigma_2$ in a sufficiently small neighbourhood of $\sigma_2=0$.
\end{proof}

% ----------------------------------------------------------------
\subsubsection{Six solutions: definition and continuation}
We remark that $\mathrm{aSR}_t(hu;d)=\mathrm{aSR}_t(hu;d+\pi)$. Therefore, Theorem \ref{thm:Y-on-chamber} allows us to construct the following six solutions $Y(z,t)$ with different resummation directions and $t$-chambers. Note that the elements in $\widetilde{\mathbb{C}^{\star}}$ is determined by both $\theta_1=\arg t$ and $\theta_2=\arg (t-1)$ together with it image $\underline{t}\in\mathbb{C}^{\star}$; therefore, given that $\underline{t_0}$ lies in the intervals $(-\infty,0)\cup(0,1)\cup(1,\infty)$, it is sufficient to determine the $t$-chamber containing $t_0$ by the arguments $\theta_1,\theta_2\in \{k\pi\mid k\in\mathbb{Z}\}$. The corresponding $t$-chamber which will be denoted $D_{\theta_1,\theta_2}$.

\begin{defi}\label{def:Six-solutions}
For suitable pairs $(\theta_1,\theta_2)$ (to be specified below), let $Y_{d,\theta_1,\theta_2}$ denote the solution obtained by the resummation of $\widehat{Y}(z,t)$ with respect to $z$ along the direction $d$, while $t$ is in the $t$-chamber $D_{\theta_1,\theta_2}$. Under this notation, we define six fundamental solutions
\begin{align*}
    &Y_{012}=Y_{d,0,0},\quad
    &&Y_{021}=Y_{d,0,\pi},\quad 
    &&Y_{201}=Y_{d,\pi,\pi},\\
    &Y_{102}=Y_{d+\pi,-\pi,-\pi},\quad
    &&Y_{120}=Y_{d+\pi,0,-\pi},\quad
    &&Y_{210}=Y_{d+\pi,0,0}.
\end{align*}    
\end{defi}

\begin{pro}\label{thm:Y_d0d1d2}
Given $d\notin \mathcal{R}_{\infty}$, for the appropriate six choices $d_0\in\{d,d+\pi\}, \, \theta_1, \theta_2\in \{0,\pm \pi\}$ above,
\begin{enumerate}
\item[(a).]
(Around $t=0$) Set $\sigma_1=z,\sigma_2=t$ and $\eta=zt$, under change of variables holds
\begin{equation}\label{eq:Y_d0d1d2-a}
Y_{d_0,\theta_1, \theta_2}(z,t)\cdot \left(\frac{t-1}{1-t}\right)^{-h\delta P}
=
\ _1W_{d_0}(\sigma_1,\sigma_2)\cdot F_{d_0+\theta_1}^{(2)}(\eta);
\end{equation}
where $_1W_{d_0}(\sigma_1,\sigma_2)$ and $F_{d_0+\theta_1}^{(2)}(\eta)$ denote the resummation of $_1\widehat{W}(\sigma_1,\sigma_2)$ w.r.t.$\sigma_1$ and $\widehat{F}^{(2)}(\eta)$ w.r.t. $\eta$ along $d_0$ and $d_0+\theta_1$ respectively.

\item[(b).]
(Around $t=\infty$) Set $\sigma_1=zt,\sigma_2=1/t$ and $\eta=z$, under change of variables holds
\begin{equation}\label{eq:Y_d0d1d2-b}
Y_{d_0,\theta_1, \theta_2}(z,t)
=
\ _2W_{d_0+\theta_1}(\sigma_1,\sigma_2)\cdot F_{d_0}^{(1)}(\eta);
\end{equation}
where $_2W_{d_0+\theta_1}(\sigma_1,\sigma_2)$ and $F_{d_0}^{(1)}(\eta)$ denote the resummation of $_2\widehat{W}(\sigma_1,\sigma_2)$ w.r.t.$\sigma_1$ and $\widehat{F}^{(1)}(\eta)$ w.r.t. $\eta$ along $d_0+\theta_1$ and $d_0$ respectively.

\item[(c).]
(Around $t=1$) Set $\sigma_1=z,\sigma_2=t-1$ and $\eta=zt-z$, under change of variables holds
\begin{equation}\label{eq:Y_d0d1d2-c}
Y_{d_0,\theta_1, \theta_2}(z,t)
=
U_{d_0}(\sigma_1,\sigma_2)\cdot X_{d_0+\theta_2}(\eta);
\end{equation}
where $U_{d_0}(\sigma_1,\sigma_2)$ and $X_{d_0+\theta_2}(\eta)$ denote the resummation of $\widehat{U}(\sigma_1,\sigma_2)$ w.r.t.$\sigma_1$ and $\widehat{X}(\eta)$ w.r.t. $\eta$ along $d_0$ and $d_0+\theta_2$ respectively.
\end{enumerate}
\end{pro}

\begin{proof}
    We prove (a). In this case let  $\sigma_1=z, \sigma_2=t$ and $\eta=zt$. Since Borel-Laplace resummation is a differential algebra homomorphism, for every fixed $t$ resum w.r.t. $z$ on both sides of the decomposition \eqref{eq:coord-Y=W1F2} gives
    \begin{equation*}
    \widehat{Y}_{d_0,d_1,d_2}(z,t)\cdot \left(\frac{t-1}{1-t}\right)^{-h\delta P}=
    \ _1\widehat{W}_{d_0}(z,t)\widehat{F}^{(2)}_{d_0}(z,t), 
    \end{equation*}
    where the subscript stands for the direction of $z$. Now, we only have to note that for fixed $t$ with $\arg t=d_1$ and $\arg (t-1)=d_2$, resumming $_1\widehat{W}(z;t)$ and $\widehat{F}^{(2)}(z;t)$ w.r.t. $z$ along direction $d_0$ are exactly resumming $_1\widehat{W}(\sigma_1,\sigma_2)$ and $\widehat{F}^{(2)}(\eta)$ w.r.t. $\sigma_1$ and $\eta$ along direction $d_0$ and $d_0+d_1$ respectively.

    Cases (b) and (c) are completely similar, as one resum the both sides of \eqref{eq:coord-Y=W2F1} and \eqref{eq:coord-Y=UX} along $d_0$, then identify with \eqref{eq:Y_d0d1d2-b} and \eqref{eq:Y_d0d1d2-c} via the corresponding coordinates and $\arg t,\arg (t-1)$ respectively.
\end{proof}

%% file: final_sections/section3/Section3.5.tex
\subsection{A proof of the RLL\texorpdfstring{$\,=\,$}{=}LLR relation}
\label{sec:3.5}

In this subsection we derive the RLL relation from the connection identities between the solutions defined in Definition \ref{def:Six-solutions}. The idea is to compare the transition matrix between $Y_{012}$ and $Y_{210}$ by two homotopic paths of analytic continuation. Each path traverses three intermediate chambers, and the transition at each step is governed by a Stokes factor of one of the three systems. Equating the total transition matrices along the two paths yields a hexagon-type identity, which reduces to the RLL relation after conjugation by diagonal factors.

\medskip

Let \(S^+_d(T^{(2)})\), \(S^+_d(T^{(1)})\) and \(S^+_d(P)\) be the Stokes matrix $S_d^+(u)$ of the system \eqref{eq:monoF2}, \eqref{eq:monoF1} and \eqref{eq:monoX} respectively. Then it holds
\begin{equation*}
    F_d^{(2)}(\eta)=F_{d+\pi}^{(2)}(\eta)\cdot S_d^+(T^{(2)});\quad
    F_d^{(1)}(\eta)=F_{d+\pi}^{(1)}(\eta)\cdot S_d^+(T^{(1)});\quad
    X_d(\eta)=X_{d+\pi}(\eta)\cdot S_d^+(P),
\end{equation*}
where the subscript denotes the respective resummation direction of $\eta$.

We compare two analytic continuations from~\(Y_{012}\) to~\(Y_{210}\):
\[
\textbf{Path I:}\quad
Y_{012}\xrightarrow{\text{path I-1}} Y_{021}\xrightarrow{\text{path I-2}} Y_{201}
\xrightarrow{\text{path I-3}} Y_{210},
\]
\[
\textbf{Path II:}\quad
Y_{012}\xrightarrow{\text{path II-1}} Y_{102}\xrightarrow{\text{path II-2}} Y_{120}
\xrightarrow{\text{path II-3}} Y_{210}.
\]
For each path segment $A\xrightarrow{l}B$ above, $l$ connects the definition domains of $A$ and $B$ in $\widetilde{\mathbb{C}^{\star}}$ and is therefore uniquely determined by Definition \ref{def:Six-solutions} (see Figure \ref{fig:paths-t-plane}). More precisely, we consider the transition obtained by continuing $A$ along $l$ into the domain of $B$ and comparing it with $B$ itself. As \textbf{Path I} and \textbf{Path II} are trivially homotopic in $\widetilde{\mathbb{C}^{\star}}$, the transition between $Y_{012}$ and $Y_{210}$ can be identified by composing the transitions along the two paths, each of which is divided into three segments by two intermediate solutions (see Figure \ref{fig:solution_connect}).

\input{final_pictures/pic_Continuation-Path-Page}

We will first compute one of the transition as an example.

\begin{lem}\label{lem:Y102-Y120}
Along the continuation path II-2, we have
\begin{equation}\label{eq:step-II2}
    Y_{102}=Y_{120}\cdot
e^{+\pi \mathi h\delta P}\,S^+_d(T^{(2)})\,e^{-\pi \mathi h\delta P}.
\end{equation}
\end{lem}

\begin{proof}
    Use (a) of Proposition \ref{thm:Y_d0d1d2}. By Definition \ref{def:Six-solutions} and \eqref{eq:Y_d0d1d2-a} one have
    \begin{align}\label{eq:Y102-and-Y120-1}
    &Y_{102}(z,t)\cdot \left(\frac{t-1}{1-t}\right)^{-h\delta P}=
    \ _1W_{d+\pi}(\sigma_1,\sigma_2)\cdot F^{(2)}_{d}(\eta)= _2W_{d}(z,t)F_{1,d}(z);\\
    \label{eq:Y102-and-Y120-2}
    &Y_{120}(z,t)\cdot \left(\frac{t-1}{1-t}\right)^{-h\delta P}=
    \ _1W_{d+\pi}(\sigma_1,\sigma_2)\cdot F^{(2)}_{d+\pi}(\eta)= _2W_{d+\pi}(z,t)F_{1,d+\pi}(z).
    \end{align}
    Along path II-2 we have $\arg (t\!-\!1)\!\in\! (-\pi,0)$, therefore $\big(\frac{t-1}{1-t}\big)^{h\delta P}\!=\!e^{-\pi \mathi h \delta P}$ by Remark \ref{rmk:multivalue_1-t}. Meanwhile since path II-2 is sufficiently close to $t=0$, Proposition \ref{pro:W-on-chamber} (a) promises $_1W_d(\sigma_1,\sigma_2)$ is single valued on all of path II-2. Combining these with \eqref{eq:Y102-and-Y120-1}-\eqref{eq:Y102-and-Y120-2} derives the identity \eqref{eq:step-II2}.
\end{proof}

\medskip\noindent
\textbf{Path I:\ \(Y_{012}\to Y_{021}\to Y_{201}\to Y_{210}\).}

For the first step where $t$ is close to $1$, use (c) of Proposition \ref{thm:Y_d0d1d2}. By Definition \ref{def:Six-solutions} and \eqref{eq:Y_d0d1d2-c} one has
\begin{align}\label{eq:Y012-and-Y021-1}
    &Y_{012}=U_d(\sigma_1,\sigma_2)\cdot X_d(\eta);\\
    \label{eq:Y012-and-Y021-2}
    &Y_{021}=U_d(\sigma_1,\sigma_2)\cdot X_{d+\pi}(\eta).
\end{align}
Since path I-1 is sufficiently close to $t=1$, Proposition \ref{pro:W-on-chamber} (c) ensures that $U_d(\sigma_1,\sigma_2)$ is single valued on all of path I-1. Combining with \eqref{eq:Y012-and-Y021-1}-\eqref{eq:Y012-and-Y021-2} derives the identity
\begin{equation}\label{eq:step-I1}
    Y_{012}=Y_{021}\cdot S^+_d(P).
\end{equation}

For the second step where $t$ is close to $0$, similar to Lemma \ref{lem:Y102-Y120}, one uses (a) of Proposition \ref{thm:Y_d0d1d2} and (a) of Proposition \ref{pro:W-on-chamber}. By Definition \ref{def:Six-solutions} and \eqref{eq:Y_d0d1d2-a} one derives
\begin{equation}\label{eq:step-I2}
Y_{021}=Y_{201}\cdot
e^{-\pi \mathi h\delta P}\,S^+_d(T^{(2)})\,e^{+\pi \mathi h\delta P}.
\end{equation}
We note here that along path I-2 $\arg (t\!-\!1)\!\in\!(0,\pi)$, therefore $\big(\frac{t-1}{1-t}\big)^{h\delta P}\!=\!e^{\pi \mathi h \delta P}$ is opposite from Lemma \ref{lem:Y102-Y120}.

For the third step where $t$ is close to $\infty$, (b) of Proposition \ref{thm:Y_d0d1d2} and (b) of Proposition \ref{pro:W-on-chamber} are used. By Definition \ref{def:Six-solutions} and \eqref{eq:Y_d0d1d2-b} one derives
\begin{equation}\label{eq:step-I3}
Y_{201}=Y_{210}\cdot S^+_d(T^{(1)}).
\end{equation}

Substituting~(\ref{eq:step-I2}) and~(\ref{eq:step-I3})
into~(\ref{eq:step-I1}) gives the total transition between $Y_{012}$ and $Y_{210}$ along path I:
\begin{equation}\label{eq:total-I}
Y_{012}=Y_{210}\cdot
S^+_d(T^{(1)})\,e^{-\pi ih\delta P}\,S^+_d(T^{(2)})\,
e^{+\pi ih\delta P}\,S^+_d(P).
\end{equation}

\medskip\noindent
\textbf{Path II:\ \(Y_{012}\to Y_{102}\to Y_{120}\to Y_{210}\).}

Step II-1 crosses~\(t=\infty\). (b) of Proposition \ref{thm:Y_d0d1d2} and Proposition \ref{pro:W-on-chamber} are used. By Definition \ref{def:Six-solutions} and \eqref{eq:Y_d0d1d2-b} one derives
\begin{equation}\label{eq:step-II1}
Y_{012}=Y_{102}\cdot S^+_d(T^{(1)}).
\end{equation}

Step II-3 crosses~\(t=1\). (c) of Proposition \ref{thm:Y_d0d1d2} and Proposition \ref{pro:W-on-chamber} are used. By Definition \ref{def:Six-solutions} and \eqref{eq:Y_d0d1d2-c} one derives
\begin{equation}\label{eq:step-II3}
Y_{120}=Y_{210}\cdot S^+_d(P).
\end{equation}

Substituting~(\ref{eq:step-II2}) and~(\ref{eq:step-II3})
into~(\ref{eq:step-II1}) gives the total transition between $Y_{012}$ and $Y_{210}$ along path II:
\begin{equation}\label{eq:total-II}
Y_{012}=Y_{210}\cdot
S^+_d(P)\,e^{+\pi \mathi h\delta P}\,S^+_d(T^{(2)})\,
e^{-\pi \mathi h\delta P}\,S^+_d(T^{(1)}).
\end{equation}

We are now ready to conclude:
\begin{pro}\label{pro:naive-RLL}
    In $\mathrm{End}(V)\otimes_s\mathrm{End}(\mathbb{C}^{m|n})\otimes_s\mathrm{End}(\mathbb{C}^{m|n})$, the following identity holds:
\begin{equation}\label{eq:naive-RLL}
S^+_d(T^{(1)})\,e^{-\pi \mathi h\delta P}\,S^+_d(T^{(2)})\,
e^{+\pi \mathi h\delta P}\,S^+_d(P)
\;=\;
S^+_d(P)\,e^{+\pi \mathi h\delta P}\,S^+_d(T^{(2)})\,
e^{-\pi \mathi h\delta P}\,S^+_d(T^{(1)}).
\end{equation}
\end{pro}

\begin{proof}
    Since path I and path II are all homotopic to a clockwise loop around $t=1$, the analytic continuation of $Y_{012}$ must coincide. Therefore, comparing \eqref{eq:total-I} and \eqref{eq:total-II} gives the desired result.
\end{proof}

\begin{lem}\label{lem:commutativity}
The following commutativity conditions hold:
\begin{equation*}
\big[S^+_d(T^{(1)}),\;\delta T^{(2)}+\delta P\big]=0,
\qquad
\big[S^+_d(T^{(2)}),\;\delta T^{(2)}+\delta P\big]=0,
\qquad
\big[S_d^+(P),\;\delta T^{(1)}+\delta T^{(2)}\big]=0.
\end{equation*}
\end{lem}

\begin{proof}
    $S_d^+(T^{(1)})$ is computed from system \eqref{eq:monoF1}, then the fact that $u^{(1)}, T^{(1)}$ and $\delta T^{(1)}$ all commutes with $\delta(T^{(2)}+P)$ gives the first commutativity. The second and third are similar.
\end{proof}

\medskip\noindent
\textbf{Proof of Theorem \ref{thm:RLL-relation}.}

Multiply \(e^{\pi \mathi h(\delta T^{(1)}+\delta T^{(2)}+\delta P)}\) to the right of both sides of \eqref{eq:naive-RLL}, Lemma \ref{lem:commutativity} allows us to move the diagonal exponentials past the Stokes factors. According to Theorem \ref{thm:super-r-matrix} $R^{-1}=S_d^+(P)e^{\pi \mathi h\delta P}$, by rearranging identities one derives
\begin{equation}\label{eq:actual-RLL}
S^+_d(T^{(1)})\,e^{\pi \mathi h\delta T^{(1)}}\cdot 
S^+_d(T^{(2)})\,e^{+\pi \mathi h\delta T^{(2)}}\cdot
R^{-1}
\;=\;
R^{-1}\cdot 
S^+_d(T^{(2)})\,e^{+\pi \mathi h\delta T^{(2)}}\cdot
S^+_d(T^{(1)})\,e^{\pi \mathi h\delta T^{(1)}}.
\end{equation}
Recall $S_{h+}(u)=S_d^+(T)\,e^{\pi i h\delta T}$, then \eqref{eq:actual-RLL} reads exactly
\begin{equation*}
R\,S^{(1)}_{h+}(u)\,S^{(2)}_{h+}(u)
\;=\;
S^{(2)}_{h+}(u)\,S^{(1)}_{h+}(u)\,R.    
\end{equation*}
which is relation \eqref{eq:rll-plus}. The same argument applied to the opposite Stokes sectors (replacing~\(d\) by~\(d+\pi\) and exchanging the roles of upper and lower Stokes matrices) yields relation \eqref{eq:rll-minus}. The mixed relation \eqref{eq:rll-mixed} is obtained by combining the two. Together, these give the three identities of Theorem~\ref{thm:RLL-relation}.
\qed

%% file: final_pictures/pic_Continuation-Path-Page.tex
\usetikzlibrary{decorations.markings}
\begin{figure}[htbp]
    \centering
    % =========================================================================
    % 左侧：图A - $C\{0,1}$万有复叠上的解析延拓路径
    % =========================================================================
    \begin{minipage}[b]{0.48\textwidth}
    \centering
        \begin{tikzpicture}[
            scale=0.8, thick,
            midarrow/.style={postaction={decorate,decoration={
                markings, mark=at position 0.6 with {\arrow{stealth}}
            }}}
        ]

    % ---- real axis (dashed) ----
    \draw[dashed, gray] (-3.5, 0) -- (7.5, 0);

    % ---- punctures at t=0 and t=1 ----
    \filldraw[fill=white, draw=black, thick] (0, 0) circle (2.2pt) node[below=3pt] {$0$};
    \filldraw[fill=white, draw=black, thick] (4, 0) circle (2.2pt) node[below=3pt] {$1$};

    % ---- start / end marker ----
    \filldraw (6, 0) circle (1.4pt) node[above right=2pt] {$Y_{012}$};
    \filldraw (6, 0) circle (1.4pt) node[below right=2pt] {$Y_{210}$};

    % =================================================================
    %  PATH I  (blue, solid)
    % =================================================================

    % I-1 : counterclockwise, passing above t = 1
    %       center at (4,0) = t=1  →  half-ellipse above
    \draw[blue!55!black, midarrow]
        (6, 0) arc (0:180:2 and 1.5)
        node[midway, above=2pt, font=\footnotesize\bfseries, blue!55!black] {I--1};

    % I-2 : counterclockwise, passing above t = 0
    %       center at (0,0) = t=0  →  half-ellipse above the axis
    \draw[blue!55!black, midarrow]
        (2, 0) arc (0:180:2 and 1.5)
        node[midway, above=2pt, font=\footnotesize\bfseries, blue!55!black] {I--2};

    % I-3 : clockwise, large radius, passing above 0 and 1 = counterclockwise, passing under \infty
    %       center at (2,0) (midway between 0 and 1)  →  half-ellipse above
    \draw[blue!55!black, midarrow]
        (-2, 0) arc (180:0:4 and 3)
        node[midway, above=3pt, font=\footnotesize\bfseries, blue!55!black] {I--3};

    % =================================================================
    %  PATH II  (green, solid)
    % =================================================================

    % II-1 : clockwise, large radius, passing below 1 and 0 = counterclockwise, passing above \infty
    %        center at (2,0)  →  half-ellipse below
    \draw[green!50!black, midarrow]
        (6, 0) arc (360:180:4 and 3)
        node[midway, below=3pt, font=\footnotesize\bfseries, green!50!black] {II--1};

    % II-2 : counterclockwise, passing below t = 0
    %        center at (0,0) = t=0  →  half-ellipse below
    \draw[green!50!black, midarrow]
        (-2, 0) arc (180:360:2 and 1.5)
        node[midway, below=2pt, font=\footnotesize\bfseries, green!50!black] {II--2};

    % II-3 : counterclockwise, passing below t = 1
    %       center at (4,0) = t=1  →  half-ellipse below the axis
    \draw[green!50!black, midarrow]
        (2, 0) arc (-180:0:2 and 1.5)
        node[midway, below=2pt, font=\footnotesize\bfseries, green!50!black] {II--3};

    % ---- chamber labels ----
    \filldraw (2, 0) circle (1.4pt) node[above left=2pt] {$Y_{021}$};
    \filldraw (2, 0) circle (1.4pt) node[below right=2pt] {$Y_{120}$}; 
    
    \filldraw (-2, 0) circle (1.4pt) node[above left=2pt] {$Y_{201}$};
    \filldraw (-2, 0) circle (1.4pt) node[below left=2pt] {$Y_{102}$};

    \end{tikzpicture}
    \caption{The paths of analytic continuation on $\widetilde{\mathbb{C}^{\star}}$, the specific branch where $\arg t,\arg (t\!-\!1)\in (-\pi,\pi)$.}
    \label{fig:paths-t-plane}
    \end{minipage}
    \hfill
    % =========================================================================
    % 右侧：图B - 六个解之间的比较关系（一个是先跃升到arg z=d+\pi，另一个后跃升）
    % =========================================================================
    \begin{minipage}[b]{0.48\textwidth}
        \centering
        \begin{tikzpicture}[
            scale=0.75, thick,
            chamber/.style={draw=black!70, rectangle, rounded corners, fill=white, minimum width=1.8cm, inner sep=5pt, font=\footnotesize, align=center},
            pathI/.style={draw, blue!55!black, ->, line width=1.2pt, >=stealth},
            pathII/.style={draw, green!50!black, ->, line width=1.2pt, >=stealth},
            jump/.style={dashed}
        ]

        % ---- 柱状背景：代表 z 相位的两个 Level ----
        % 右侧柱子：初始相平面
        \fill[blue!5, rounded corners] (1.1, -3.8) rectangle (3.9, 3.5);
        \node[blue!60!black, align=center, font=\scriptsize] at (2.5, -3.4) {$\arg z = d$};

        % 左侧柱子：跃升相平面
        \fill[green!5, rounded corners] (-3.9, -2.8) rectangle (-1.1, 4.5);
        \node[green!60!black, align=center, font=\scriptsize] at (-2.5, 4.1) {$\arg z = d+\pi$};

        % ---- 拓扑格点 (右侧列 - Level 0) ----
        \node[chamber] (Y012) at (2.5, -2.2) {$Y_{012}$: to be\\continued};
        \node[chamber] (Y021) at (2.5, 0.3) {$Y_{021}$};
        \node[chamber] (Y201) at (2.5, 2.8) {$Y_{201}$};

        % ---- 拓扑格点 (左侧列 - Level 1) ----
        \node[chamber] (Y102) at (-2.5, -2.1) {$Y_{102}$};
        \node[chamber] (Y120) at (-2.5, 0.4) {$Y_{120}$};
        \node[chamber] (Y210) at (-2.5, 2.9) {$Y_{210}$: to be\\compared};

        % ---- PATH I 的拓扑流动 (蓝色，先跨越 t，再跨越无穷大) ----
        \path[pathI] (Y012) edge node[midway, right=2pt] {\scriptsize $t\circlearrowleft 1$} (Y021);
        \path[pathI] (Y021) edge node[midway, right=2pt] {\scriptsize $t\circlearrowleft 0$} (Y201);
        \path[pathI, jump] (Y201) edge node[midway, above=1pt, sloped] {\scriptsize $t \circlearrowleft \infty$ ($z$-shift)} (Y210);

        % ---- PATH II 的拓扑流动 (绿色，先跨越无穷大，再跨越 t) ----
        \path[pathII, jump] (Y012) edge node[midway, below=1pt, sloped] {\scriptsize $t \circlearrowleft \infty$ ($z$-shift)} (Y102);
        \path[pathII] (Y102) edge node[midway, left=2pt] {\scriptsize $t\circlearrowleft 0$} (Y120);
        \path[pathII] (Y120) edge node[midway, left=2pt] {\scriptsize $t\circlearrowleft 1$} (Y210);

        \end{tikzpicture}
        \caption{The intermediate solutions as anchor points of comparison along the two continuation.}
        \label{fig:solution_connect}
    \end{minipage}

    \label{fig:overall_solution}
\end{figure}

%% file: final_sections/4Representation.tex
\section{Two realizations of the quantum supergroup \texorpdfstring{$U_q(\mathfrak{gl}(m|n))$}{Uq(gl(m|n))}}\label{sec:4}

In this section we recall the definition of quantum supergroup $U_q(\mathfrak{gl}(m|n))$ from both the Drinfeld-Jimbo approach \cite{Drinfeld, Jimbo2} and the Faddeev--Reshetikhin--Takhtajan (FRT) formalism \cite{FRT}. The two Hopf (super)algebras are isomorphic via a super analog of the Ding-Frenkel theorem \cite{Ding-Frenkel1993}. This is further applied to prove Theorem \ref{thm:intro-Rep}. For more details, we refer to \cite{Yamane1994, Manin1989}, and in Section \ref{sec:4.1} we mainly follow the context in \cite{Zhang2016}.

\subsection{The quantum supergroup in two approaches}\label{sec:4.1}

\begin{defi}\label{def:DJ-quantumSupergroup}
Take $q$ not a root of unity. The quantum supergroup $U_q(\mathfrak{gl}(m|n))$ is the superalgebra generated by $e_i,f_i,K_j^{\pm}$ for $1\leq i\leq m+n-1$ and $1\leq j\leq m+n$. The generators have parities $|e_m|=|f_m|=1$, and $|e_i|=|f_i|=|K_j|=0$ otherwise. They are subject to the following relations.
\begin{enumerate}
\item The Cartan generators $K_j$ satisfy
\begin{equation*}
    K_jK_j^{-1}=K_j^{-1}K_j=1,\quad 
    K_je_iK_j^{-1}=q^{\delta_{ji}-\delta_{j,i+1}}e_i,\quad 
    K_jf_iK_j^{-1}=q^{\delta_{j,i+1}-\delta_{ji}}f_i.
\end{equation*}
\item The super commutators between $e_i$ and $f_j$ are equal to
\begin{equation*}
    [e_i,f_j]=\delta_{ij}\frac{K_i^{d_i}K_{i+1}^{-d_{i+1}}-K_i^{-d_i}K_{i+1}^{d_{i+1}}}{q_i-q_i^{-1}}.
\end{equation*}
\item let $[x,y]_r=xy-r(-1)^{|x||y|}yx$ denote the deformation of super commutator, it holds

The bosonic relation:
\begin{align*}
    [e_i, e_j] &= 0, \quad [f_i, f_j] = 0 \quad \text{if } |i-j| \neq 1, \\
    [e_i, [e_i, e_j]_{q^{-1}}]_q &= 0, \quad [f_i, [f_i, f_j]_{q^{-1}}]_q = 0 \quad \text{if } |i-j| = 1, i\neq m.
\end{align*}
The fermionic relation if $m,n>1$:
\begin{align*}
    \big[\big[[e_{m-1}, e_m]_q, e_{m+1}\big]_{q^{-1}}, e_m\big] &= 0, \\
    \big[\big[[f_{m-1}, f_m]_q, f_{m+1}\big]_{q^{-1}}, f_m\big] &= 0.
\end{align*}
\end{enumerate}
Here $d_i=(-1)^{\bar{i}}$ is $1$ for $1\le i\le m$ and $-1$ for $m<i\leq m+n$. Factor $q_i$ is taken to be $q^{d_i}$. Particularly, let $i=j=m$ in the bosonic relation we have $e_m^2=f_m^2=0$. 
\end{defi}
The quantum supergroup $U_q(\mathfrak{gl}(m|n))$ has a Hopf superalgebra structure. Since we only need the product structure, we omit the coproduct here.

\begin{defi}\label{def:FRT-quantumSupergroup}
Let $R$ be the standard super $R-$matrix valued in $\mathrm{End}(\mathbb{C}^{m|n})\otimes\mathrm{End}(\mathbb{C}^{m|n})$, see \eqref{eq:standard-r-matrix}. Then the super FRT algebra $U(R)$ is the unital associative superalgebra generated by elements $\{l_{ij}^{(+)},\,l_{ji}^{(-)}\}_{1\le i\le j\le m+n}$ of parity $|l_{ij}^{(\pm)}|=\bar i+\bar j$, subject to the following relations. Set
\begin{equation*}\label{eq:Lpm-def}
L_{+}=\sum_{i,j=1}^{m+n}l_{ij}^{(+)}\otimes E_{ij},\qquad
L_{-}=\sum_{i,j=1}^{m+n} l_{ij}^{(-)}\otimes E_{ij},
\end{equation*}
where we declare $l_{ij}^{(+)}=0$ for $i>j$ and $l_{ij}^{(-)}=0$ for $i<j$.
Then the defining relations are the following identities in $U(R)\otimes \mathrm{End}(\mathbb{C}^{m|n})\otimes \mathrm{End}(\mathbb{C}^{m|n})$:
\begin{subequations}\label{eq:FRT-relation}
\begin{align}
\label{eq:FRT-pm}
R^{12} L_{\pm}^{(1)} L_{\pm}^{(2)} &= L_{\pm}^{(2)} L_{\pm}^{(1)} R^{12},\\
\label{eq:FRT-mixed}
R^{12} L_{+}^{(1)} L_{-}^{(2)} &= L_{-}^{(2)} L_{+}^{(1)} R^{12},
\end{align}    
\end{subequations}
where the superscripts indicate the tensor factor in $\mathrm{End}(\mathbb{C}^{m|n})\otimes\mathrm{End}(\mathbb{C}^{m|n})$ on which the respective operator acts, and all products are computed in the super tensor product with Koszul signs. Additionally,
\begin{equation}\label{eq:FRT-unitarity}
l_{ii}^{(+)} l_{ii}^{(-)} = l_{ii}^{(-)} l_{ii}^{(+)} = 1,
\qquad 1\le i\le m+n.
\end{equation}
\end{defi}

\begin{pro}[c.f.\cite{FanHouShi1997, Zhang1997}]\label{prop:superDing-Frenkel}
There exists an isomorphism of Hopf superalgebras
$\Phi: U_q(\mathfrak{gl}(m|n))\xrightarrow{\cong} U(R)$.
On the generators $\{e_i,f_i\}_{i=1}^{m+n-1},\{K_j\}_{j=1}^{m+n}$ of $U_q(\mathfrak{gl}(m|n))$, the isomorphism is given by
\begin{equation*}
\Phi(K_j)=\big(l_{jj}^{(+)}\big)^{d_j}=\big(l_{ii}^{(-)}\big)^{-d_j},\qquad
\Phi(e_i)=\frac{l_{ii}^{(-)}l_{i,i+1}^{(+)}}{q_i^{-1}-q_i},\qquad
\Phi(f_i)=\frac{l_{i+1,i}^{(-)}l_{ii}^{(+)}}{q_i-q_i^{-1}}.
\end{equation*}
\end{pro}

\subsection{The proof of Theorem \ref{thm:intro-Rep}}

Now we are ready to prove Theorem \ref{thm:intro-Rep}.

\begin{proof}
By Proposition~\ref{prop:triangularity-sec2} and Definition~\ref{def:renorm-Stokes}, the Stokes supermatrices $S_{h\pm}(u)$ are upper and lower triangular with respect to the dominance ordering $\sigma_d$. That is $(S_{h+}(u))_{ij}=0$ for $i>j$ and $(S_{h-}(u))_{ij}=0$ for $i<j$. Moreover, the diagonal entries satisfy $S_{h+}(u)_{ii}\,S_{h-}(u)_{ii}=1$ which is exactly \eqref{eq:FRT-unitarity}. Now Theorem~\ref{thm:RLL-relation} states that $S_{h+}(u)$ and $S_{h-}(u)$ satisfy the (super) $RLL$ relations~\eqref{eq:FRT-pm}--\eqref{eq:FRT-mixed}. Therefore, the assignment $l_{ij}^{(\pm)}\mapsto S_{h\pm}(u)_{i,j}$ extends to an algebra homomorphism $\varphi: U(R)\to\mathrm{End}(V)$.

Eventually, composing with the isomorphism $\Phi$ in Proposition \ref{prop:superDing-Frenkel}, we obtain
\begin{equation*}
\mathcal{S}_q(u):=\varphi\circ\Phi:\;
U_q(\mathfrak{gl}(m|n))\longrightarrow \mathrm{End}(V)
\end{equation*}
which is the desired representation.
\end{proof}

\subsection*{Acknowledgements}
\noindent
We would like to thank Jinghong Lin, Qian Tang and Xianda Wang for their useful comments on the paper.
The authors are supported by the National Key Research, Development Program of China 
(No.~2021YFA1002000). Z.Wang is also supported by Beijing Natural Science Foundation (No.~QY25087).

\input{final_bibtex/Biblio}
\AddressesLi

\AddressesWang

\Addresses